\documentclass[leqno]{article}
\usepackage{setspace} 
\usepackage{geometry}
\usepackage{fancyhdr}
\usepackage{cite}
\usepackage{graphicx}
\usepackage{epstopdf}
\usepackage[caption=false]{subfig}
\usepackage{tikz} % make some pictures
\usetikzlibrary{positioning, shapes.geometric}
\usetikzlibrary{arrows}

\usepackage{color}
\usepackage{xcolor}
\usepackage{hyperref}
\hypersetup{hidelinks} % hide the colors and box for links
\usepackage{indentfirst}
\usepackage{url}
\usepackage[normalem]{ulem}
\let\isout\sout \renewcommand{\sout}[1]{\ifmmode\text{\isout{\ensuremath{#1}}}\else\isout{#1}\fi}

\ifpdf
 \DeclareGraphicsExtensions{.eps,.pdf,.png,.jpg}
\else
 \DeclareGraphicsExtensions{.eps}
\fi

\usepackage{amsmath,amssymb,amsthm,amsfonts}
\usepackage{algorithm}
\usepackage{algorithmic}
\newtheorem{theorem}{Theorem}[section]
\newtheorem{lemma}[theorem]{Lemma}
\newtheorem{remark}[theorem]{Remark}

\numberwithin{theorem}{section}
\numberwithin{equation}{section}
\numberwithin{figure}{section}
\numberwithin{algorithm}{section}
\usepackage{mathrsfs}
\usepackage{bm}
\usepackage{enumerate}

\newcommand{\TV}[1]{\left\Vert{\hskip -2.7pt}\left\vert #1 \right\vert{\hskip -2.7pt}\right\Vert} % Triple bar norm
\providecommand{\Div}{\operatorname{div}} % Divergence
\providecommand{\supp}{\operatorname{supp}} % support
\newcommand{\wh}{\widehat}

\providecommand{\diam}{\operatorname{diam}}
\newcommand{\diff}{{\rm d}}
\newcommand{\mcE}{\mathcal{E}}
\newcommand{\mcA}{\mathcal{A}}
\newcommand{\mcB}{\mathcal{B}}
\newcommand{\mcD}{\mathcal{D}}
\newcommand{\mcL}{\mathcal{L}}
\newcommand{\mcP}{\mathcal{P}}
\newcommand{\mcT}{\mathcal{T}}
\newcommand{\OchO}{\Omega\cup\hat\Omega}
\newcommand{\mbi}{\mathbf{i}}
\newcommand{\bbR}{\mathbb{R}}
\newcommand{\hu}{\hat{u}}
\newcommand{\bsr}{\boldsymbol{r}}
\newcommand{\bsq}{\boldsymbol{q}}

\newcommand{\Jp}[1]{\left[ #1 \right]} % [ * ]

\newcommand{\Sp}[1]{\left( #1 \right)} % ( * )

\newcommand{\bggS}[1]{\bigg( #1 \bigg)}

\newcommand{\Dp}[1]{\left\langle #1 \right\rangle} % < * >

\newcommand{\norm}[1]{\left\Vert #1 \right\Vert}

\newcommand{\eq}[1]{\begin{align}#1\end{align}}
\newcommand{\eqn}[1]{\begin{align*}#1\end{align*}}

\title{Higher-order FEM and CIP-FEM for Helmholtz equation with high wave number and perfectly matched layer truncation} 
\author{
Yonglin Li\thanks{Department of Applied Mathematics, The Hong Kong Polytechnic University, Hung Hom, Hong Kong. The research of Yonglin Li was partially supported by the CAS AMSS-PolyU Joint Laboratory of Applied Mathematics.
Email address: liyonglin@smail.nju.edu.cn}
\,\,\,\,\mbox{and}\,\,\,
Haijun Wu\thanks{Department of Mathematics, Nanjing University, Jiangsu 210093, People's Republic of China. 
Email address: hjw@nju.edu.cn}
}

\date{}

\begin{document}

\maketitle

\vspace{-10pt}

\begin{abstract}
  The high-frequency Helmholtz equation on the entire space is truncated into a bounded domain using the perfectly matched layer (PML) technique and subsequently, discretized by the higher-order finite element method (FEM) and the continuous interior penalty finite element method (CIP-FEM). 
  By formulating an elliptic problem involving a linear combination of a finite number of eigenfunctions related to the PML differential operator, a wave-number-explicit decomposition lemma is proved for the PML problem, which implies that the PML solution can be decomposed into a non-oscillating elliptic part and an oscillating but analytic part.
  The preasymptotic error estimates in the energy norm for both the $p$-th order CIP-FEM and FEM are proved to be $C_1(kh)^p + C_2k(kh)^{2p} +C_3\mathcal{E}^{\rm PML}$ under the mesh condition that $k^{2p+1}h^{2p}$ is sufficiently small, where $k$ is the wave number,  $h$ is the mesh size, and $\mathcal{E}^{\rm PML}$ is the PML truncation error which is exponentially small. In particular, the dependences of coefficients $C_j~(j=1,2)$ on the source $f$ are improved.
  Numerical experiments are presented to validate the theoretical findings, illustrating that the higher-order CIP-FEM can greatly reduce the pollution errors.

\end{abstract}

{\bf Key words:} 
Helmholtz equation, high wave number, perfectly matched layer, higher-order finite element methods, preasymptotic error estimates.

{\bf AMS subject classifications:}
65N12, %Stability and convergence of numerical methods
65N15, %Error bounds
65N30, %Finite elements, Rayleigh-Ritz and Galerkin methods, finite methods
78A40  %Wave and radiation

\section{Introduction}\label{sec:Introduction}

In this paper, we consider the high-frequency acoustic scattering problem on the entire space, which is a significant challenge in computational physics and can be described by the Helmholtz equation with the Sommerfeld radiation boundary condition, i.e.,
\begin{alignat}{2}
  -\Delta u - k^2 u &= f &\quad &\text{in}\; \bbR^d, \label{eq:Helm}\\
  \left\vert \frac{\partial u}{\partial r} - \mbi k u \right\vert &= o\Big(r^{\frac{1-d}{2}}\Big) &\quad &\text{for} \; r :=\vert x\vert \rightarrow\infty,\label{eq:Somm}
\end{alignat}
where $d \in \{1,2,3\}$ denotes the dimension of the space, $k$ is the wave number, and $f\in L^2(\bbR^d)$ is a given source satisfying $f=0$ outside a bounded convex domain $\Omega$. For simplicity, we suppose $\Omega = \mathcal{B}_R$, where $\mathcal{B}_r$ denotes the ball with center at the origin and radius $r$. Denote the boundary of $\Omega$ by $\Gamma := \partial\Omega$. We also assume that $k \geq 1$ since we are considering the high-frequency problem. 

To solve the Helmholtz problem \eqref{eq:Helm}--\eqref{eq:Somm} numerically, one should truncate the unbounded domain into a bounded one by imposing some artificial boundary condition on the truncation boundary. 
A highly effective and widely employed method for this purpose is the perfectly matched layer (PML) technique. Originally introduced by B{\'e}renger in \cite{berenger94}, the PML technique has undergone significant development and adaptation for various wave propagation problems in engineering and physics \cite[etc.]{bp06,bp13,bw05,cjm97,ty98,cl05,cm98,cw03,cw08,cw94,cx13,cz10,ls01,ls98}.
The key idea of the PML method is to construct an absorbing layer outside the concerned domain $\Omega$, which can strongly absorb the outgoing waves entering the layer. 
In view of this, one can truncate the scattered field by a homogeneous Dirichlet boundary condition after an appropriate distance away from $\Omega$, say, at $r=\hat R$ for some $\hat R > R$. The PML solution is then solved in the resulting bounded domain, which is denoted by $\mcD:=\mcB_{\hat R}$, as depicted in Figure \ref{fig:PML}.
Previous research has consistently demonstrated that the PML solution exhibits exponential convergence towards the radiation solution as the layer width or PML parameter tends to infinity, as shown in studies such as 
\cite{cl05,bw05,bp06,bp13,lw19,gls21,cgn+22}. 
%\cb{Notably, a wide variety of Helmholtz scattering problems by sound-soft, sound-hard, and penetrable obstacles is considered in recent work \cite{gls21}, which proved that the PML solution is exponentially close to the true solution in both wave number and PML width.}
\begin{figure}[t]
  \centering
  \begin{tikzpicture}
    \draw [fill=gray!10] (0,0) circle [radius=1.5];
    \draw (0,0) circle [radius=2.2];
    \draw [-,thick] (0,0) -- (1.5,0);
    \draw [-,thick] (0,0) -- (1.5,1.61);
    \node [scale=0.9] at (-0.1,-0.3) {$\Omega$};
    \node [scale=0.9] at (2.1,-1.4) {$\mathcal{D}$};
    \node [scale=0.9] at (1.2,0.2) {$R$};
    \node [scale=0.9] at (1.1,1.6) {$\hat R$};
    \node [scale=0.9] at (0,-1.8) {PML};
  \end{tikzpicture}
  \caption{Setting of the PML in two dimensions.}
  \label{fig:PML}
\end{figure}
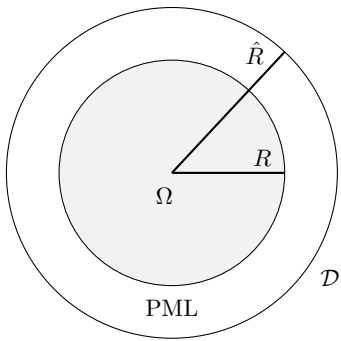

In the bounded domain $\mcD$, the numerical solution of the PML problem is accomplished through discretization methods like the finite element method (FEM).
However, it is well-acknowledged that the standard FEM for wave propagation problems with high wave numbers suffers from the pollution effect, see \cite{ib95,ib97,bs00,dbb99}. 
As mentioned in \cite[etc.]{wu14,zw13,dw15}, we use the phrase ``asymptotic'' refers to the estimate without pollution error and the phrase ``preasymptotic'' refers to the estimate with non-negligible pollution effect. 
For the asymptotic error estimates we refer to 
\cite{ib95,ib97,ms10,ms11} which proved that the $p$-th order FEM is pollution free if $k^{p+1}h^{p}$ is small enough or $p\gtrsim \ln k$ and ${kh}/p$ is sufficiently small, where $h$ represents the mesh size. 
As for the preasymptotic error estimates, a modified duality argument was first proposed in \cite{zw13} to prove hat the pollution error is $O(k^{2p+1}h^{2p})$ if $k^{p+2}h^{p+1}$ is sufficiently small. 
Later, the literature \cite{dw15} relaxed the mesh condition to $k^{2p+1}h^{2p}$ being sufficiently small by utilizing the negative-norm estimates.
To our knowledge, this is the best result by far for the preasymptotic error estimate of higher-order FEM. 
Numerous approaches have emerged over the past two decades to reduce the pollution errors, including $hp$-FEM \cite{ib97,ms10,ms11,cgn+22}, CIP-FEM \cite{wu14,zw13,dw15,lw19,lzz20}, discontinuous Galerkin method (DG) \cite{fw09,fw11,mps13,zw21,gm11}, Trefftz methods \cite{hm07,ghp09,hmp16,hmp16a,hmp11,hmp14,gh14,km18}, and multiscale methods \cite{bgp17,peterseim16,gp15}. In this paper, we would like to introduce the higher-order CIP-FEM which offers significant advantages in reducing pollution errors.

It is noteworthy that most error analyses in the existing literature, including most of the references mentioned above, focus on the Helmholtz equation \eqref{eq:Helm} with the impedance boundary condition or the Dirichlet-to-Neumann (DtN) boundary condition.
There has been a paucity of research specifically addressing the Helmholtz problem with PML boundary condition.
In our previous work \cite{lw19}, we considered the Helmholtz equation truncated by the PML with constant PML parameter, and gave a wave-number-explicit analysis of the $h$-version of both the linear FEM and CIP-FEM, showing that the preasymptotic error for PML problem is the same as that for the Helmholtz problems with the impedance or DtN boundary, i.e., the estimate $C_1kh+C_2k^3h^2$ is obtained provided $k^3h^2$ is small enough.
In \cite{cgn+22}, the authors explored the Helmholtz equation with a smooth, star-shaped obstacle and PML truncation, offering a preasymptotic error estimate $C_1 kh + C_2 k^{2p+1}h^{2p}$ for the higher-order FEM under the condition that $k^{p+2}h^{p+1}$ is small enough. 
The findings presented in \cite{cgn+22} heavily depend on the assumption that the PML parameter satisfied $\sigma \sim 1/k$, and they are not applicable when $\sigma$ significantly exceeds $1/k$. 
Recently, the authors of \cite{gs23} extended the ``elliptic-projection'' argument, originally proposed in \cite{zw13} for Helmholtz problem with impedance boundary, to a broad range of abstract Helmholtz-type problems. They established preasymptotic error estimates for $h$-FEM by assuming that $(hk)^{2p}C_{\rm sol}$ is sufficiently small, with $C_{\rm sol}$ representing the norm of the solution operator from $L^2$ to $H^2$. Additionally, it is proved in \cite{gls21} that $C_{\rm sol} \lesssim k$ when the PML scaling function belongs to a $C^3$ class.

In this paper, a continuation of our previous work \cite{lw19}, we study the higher-order FEM and CIP-FEM for the Helmholtz equation truncated by PML technique with constant PML parameter, which is not covered by the theories in \cite{gls21} and \cite{gs23}. 
The major contributions of this paper are listed as follows.
\begin{enumerate}[(i)]
  \item Inspired by the recent work \cite{gs23}, we re-prove the decomposition lemma for the truncated PML solution by using a novel technique; see Lemma \ref{lem:decomp}. This method, based on a truncation on the eigenfunction expansion of the PML differential operator, offers a simpler approach compared to the previous methods like the Fourier technique in \cite{ms10,ms11}, the semi-classical method in \cite{gls+21,gls+22,lsw22}, and the iterative technique in \cite{cn20a,cgn+22,gs23}. 
  \item We establish the preasymptotic error estimates in the energy norm of the form $C_1(kh)^p + C_2k(kh)^{2p} +C_3\mathcal{E}^{\rm PML}$ for both $p$-th order FEM and CIP-FEM under the mesh condition that $k^{2p+1}h^{2p}$ is sufficiently small, where $\mathcal{E}^{\rm PML}$ is the PML truncation error which is exponentially small (see \eqref{eq:err-PML} below). 
  This enhances the results of \cite{cgn+22} and extends those of \cite{lw19}; see Theorem \ref{thm:err:FEM}. Notably, this marks the first error estimate for higher-order CIP-FEM applied to the Helmholtz equation with PML.
  \item In our error estimate, we explicitly illustrate the dependences on the source $f$ for both the coefficients $C_1$ and $C_2$. 
  The dependences in our estimate are weaker than the previous result (see e.g., \cite{gs23,dw15}). 
  In particular, we remark that the coefficient of the pollution term $O(k^{2p+1}h^{2p})$ depends only on the $L^2$-norm of $f$; see Remark \ref{rmk:err:FEM}.
\end{enumerate}

The outline of this paper is as follows. In Section \ref{sec:preliminary}, we introduce the PML model problem and the existing convergence result for PML truncation. Section \ref{sec:PML} presents a rigorous analysis of the wave-number-explicit decomposition lemma for the PML problem. Section \ref{sec:CIP-FEM} is dedicated to the preasymptotic error estimate for CIP-FEM, as well as FEM. In Section \ref{sec:test}, we provide some numerical examples to validate our theoretical findings.

\section{Preliminaries} \label{sec:preliminary}

\subsection{Notations}
Throughout the paper, $C$  is used to denote a generic positive constant which is independent of $h$, $k$, and $f$, but may depend on $p$. We also use the shorthand notations $A \lesssim B$ and $B\lesssim A$ for the inequality $A \leq C B$ and $B\geq C A$. $A\eqsim B$ is a shorthand notation for the statement $A\lesssim B$ and $B\gtrsim  A$.  The standard space, norm, and inner product notation are adopted (see, e.g., \cite{bs08,ciarlet78}). In particular, we denote by $(\cdot,\cdot)_G$ the inner products on the complex-valued Hilbert spaces $L^2(G)$ for any domain $G\subset\bbR^d$. For simplicity, we denote $(\cdot,\cdot) = (\cdot,\cdot)_\mcD$. 
Moreover, for any disjoint domains $G_1$ and $G_2$, the piecewise Sobolev space is defined by
\[
  H^m(G_1\cup G_2) := \{v: v|_{G_1} \in H^m(G_1), \;\; v|_{G_2} \in H^m(G_2)\}
\]
with the norm 
\[
  \|v\|_{H^m(G_1\cup G_2)} := \big(\|v\|_{H^m(G_1)}^2 + \|v\|_{H^m(G_2)}^2\big)^{1/2}
\]
and semi-norm
\[
  |v|_{H^m(G_1\cup G_2)} := \big(|v|_{H^m(G_1)}^2 + |v|_{H^m(G_2)}^2\big)^{1/2}.
\]

\subsection{The approximate PML equation}
It is well known that the PML equation can be viewed as a complex coordinate stretching of the original Helmholtz scattering problem \eqref{eq:Helm} (see, e.g., \cite{cm98}). For simplicity, we consider the radial PML with constant absorbing coefficient. Let
\[
  \tilde{r} := \int_0^r \alpha(s) \diff s = r \beta(r) \;\;\mbox{with}\;\; \alpha(r) = 1 + \mbi \sigma(r) \;\;\mbox{and}\;\;\beta(r) = 1 + \mbi\delta(r),
\]
where 
\begin{equation} \label{eq:medium}
  \sigma (r)=
  \begin{cases}
  0, & r \leq R, \\
  \sigma_0,  & r>R,
  \end{cases}
  \qquad 
  \delta(r)=
  \begin{cases}
  0, & r \leq R, \\
  \frac{\sigma_0(r-R)}{r},  & r>R, 
  \end{cases}
\end{equation}
with the PML parameter $\sigma_0 > 0$ being a constant. 
We remark that the medium property $\sigma$ is set as a constant in PML here to simplify the theoretical analysis, while it is possible to employ a variable PML medium properties in practice.
By replacing the radial coordinate $r$ with $\tilde{r}$ for the Helmholtz equation \eqref{eq:Helm} and truncating the PML solution at some $\hat R > R$, we obtain the approximate PML problem, which is written by
\begin{alignat}{2}
  -\Div (A\nabla \hu) - B k^2 \hu &= f &\quad&\mbox{in } \mathcal{D}:=\mathcal{B}_{\hat R}, \label{eq:PML1}\\
  \hu&= 0 &&\mbox{on } \hat\Gamma:=\partial \mathcal{D}, \label{eq:PML2}
\end{alignat}
where the matrix function $A(x)$ and scalar function are given by
\begin{align*}
  A=HDH^{\rm T},\quad  B=\alpha(r)\beta^{d-1}(r),
\end{align*}
with the matrices $D$ and $H$ given by 
\[
\begin{array}{cc}
D=\frac{1}{\alpha(x)},~H=1 & \text{for } d=1,\\
D=\begin{pmatrix}
  \frac{\beta(r)}{\alpha(r)} & 0 \\
  0 & \frac{\alpha(r)}{\beta(r)}	
  \end{pmatrix},~
H=\begin{pmatrix}
  \cos\theta & -\sin\theta \\
  \sin\theta & \cos\theta
  \end{pmatrix} & \text{for } d=2,\\
D=\begin{pmatrix}
  \frac{\beta^2(r)}{\alpha(r)} & 0    & 0    \\
  0                       & \alpha(r) & 0    \\
  0                       & 0         & \alpha(r)
  \end{pmatrix},~
H=\begin{pmatrix}
  \sin\theta\cos\varphi & \cos\theta\cos\varphi & -\sin\varphi \\
  \sin\theta\sin\varphi & \cos\theta\sin\varphi & \cos\varphi \\
  \cos\theta            & -\sin\theta           & 0 
  \end{pmatrix} & \text{for } d=3\,.
\end{array}
\]
Denote the thickness of PML by $L:=\hat R - R$. The variational formulation of the PML problem \eqref{eq:PML1}--\eqref{eq:PML2} reads as: find $\hu \in H_0^1(\mcD)$ such that
\begin{equation}\label{eq:vp}
  a(\hu, v) = (f, v) \quad\forall  v \in H_0^1(\mcD),
\end{equation}
where the sesquilinear form $a$ is defined by
\begin{equation}
  a(u,v):= (A\nabla u, \nabla v) - k^2(Bu, v). \label{eq:auv}
\end{equation}
Some direct calculations (cf. \cite[\S 2.1]{jlw+22}) yield
\begin{alignat}{2}
  |(A \bsq, \bsr)| &\lesssim \|\bsq\|_{L^2(\mcD)} \|\bsr\|_{L^2(\mcD)},\;\; \Re (A \bsr, \bsr) = \Re (\bsr, A \bsr) \eqsim \|\bsr\|_{L^2(\mcD)}^2 &\quad & \forall  \bsq, \bsr \in [L^2(\mcD)]^d, \label{eq:Aqr} \\
  |(B v,w)| &\lesssim \|v\|_{L^2(\mcD)} \|w\|_{L^2(\mcD)},\;\; %\Re (Bv,v) = \Re (v,Bv) \geq 
  (1-\tilde c)\|v\|_{L^2(\mcD)}^2 \leq \Re (Bv,v) \leq \|v\|_{L^2(\mcD)}^2 &\quad &\forall  v,w \in L^2(\mcD), \label{eq:Bvw}
\end{alignat}
where $\tilde{c} \geq 0$ is a constant depending only on $\sigma_0$. 
% Thus, we have
% \begin{equation}\label{eq:Avv}
%   \|\bsr\|_{L^2(\mcD)} \eqsim \Re (A\bsr, \bsr)^{1/2} 
%   % \quad\mbox{and}\quad \|v\|_{L^2(\mcD)} \eqsim \Re(Bv, v)^{1/2}
%   \quad \forall  \bsr \in [L^2(\mathcal{D})]^d,\; v \in L^2(\mcD).
% \end{equation}
For further analysis, we define the energy norm for $v \in H^1(\mcD)$ by
\begin{equation}
  \TV{v} := \big(\|\nabla v\|_{L^2(\mcD)}^2 + k^2 \|v\|_{L^2(\mcD)}^2\big)^{1/2}. \label{eq:e-norm}
\end{equation}
By the continuities in \eqref{eq:Aqr}--\eqref{eq:Bvw}, there holds
\begin{align}\label{est:a:cont}
  |a(v,w)| \lesssim \TV{v} \TV{w} \quad\forall  v, w \in H^1(\mcD).
\end{align}

The fundamental analysis shows that the PML solution $\hu$ converges exponentially to the scattering solution $u$ when the thickness $L$ of PML or the PML parameter $\sigma_0$ tends to infinity. The following convergence result is proved in \cite[Theorem 3.7]{lw19}.
\begin{lemma}\label{lem:conv}
  Suppose $R \eqsim \hat R\eqsim 1$. Assume that $k\sigma_0 L \geq 1$ for $d=1$ and that
  \begin{equation}\label{PML:cond}
    kR \ge 1 \;\;\mbox{and}\;\; k\sigma_0 L \geq \max\{2kR+\sqrt{3} k L, \, 10\} \quad \mbox{for}\; d\in\{2,3\}.
  \end{equation}
  The solution $\hu \in H_0^1(\mcD)$ to the PML problem \eqref{eq:vp} uniquely exists and satisfies
  \[
    k\|u-\hu\|_{L^2(\Omega)} + \|\nabla(u-\hu)\|_{L^2(\Omega)} \lesssim \mcE^{\rm PML} \|u\|_{{H^{1/2}(\Gamma)}},
  \]
  where the coefficient $\mathcal{E}^{\rm PML}$ is exponentially small  and given by
  \begin{equation}\label{eq:err-PML}
    \mcE^{\rm PML} = 
    \begin{cases}
      k e^{-2 k \sigma_0 L}, & d=1, \\
      k^5 e^{-2 k \sigma_0 L\big(1-\frac{R^2}{\hat R^2+\sigma_0^2 L^2}\big)^{1/2}}, & d\in\{2,3\}.
    \end{cases}
  \end{equation}
\end{lemma}

\section{Analyses of PML problem}\label{sec:PML}
For convenience, we denote the PDE operator to the truncated PML problem by
$$\mcL w := -\Div (A\nabla w)-Bk^2 w$$ 
for any $w \in H_0^1(\mcD)$. Clearly,
\eqn{\Dp{\mcL w,v}=a(w,v)\quad\forall w,v\in H_0^1(\mcD).}
In this section, we first consider the PML problem with a source $g \in L^2(\mcD)$ satisfying $\supp g \subset\mathcal{D}$, which is
\begin{align}
  \mcL w =g \;\;\text{in}\; \mcD; \quad w=0  \;\;\text{on}\; \hat\Gamma. \label{eq:PML3} 
\end{align}

\subsection{Stability and higher regularity}
The following stability result is proved in \cite[Theorem 3.1, Corollaries 3.4 and 3.9]{lw19}.
\begin{lemma}\label{lem:stab}
  Under the conditions of Lemma \ref{lem:conv},
  let $w$ be the solution to \eqref{eq:PML3}, 
  it holds
  \begin{equation}\label{eq:stab}
    k \|w\|_{L^2(\mcD)} + \|w\|_{H^1(\mcD)} + k^{-1} \|w\|_{H^2(\Omega \cup \hat\Omega)} \leq C_{\rm stab} \|g\|_{L^2(\mcD)},
  \end{equation}
  where the stability constant $C_{\rm stab}$ is independent of the wave number $k$.
\end{lemma}

To derive the higher regularity estimates for the PML problem, we first introduce the following result for the elliptic equation (see \cite[Theorem 1.1]{babuska70}).

\begin{lemma}\label{lem:high-regularity}
  Let $z$ be the solution to the following elliptic equation
  \[
    -\Div (A \nabla z) = g \;\;\mbox{in }\mcD;\quad
    z = 0 \;\;\mbox{on }\hat\Gamma.
  \]
  Suppose that $g \in H^m(\OchO)$, then $z \in H^{m+2}(\OchO)$ and satisfies the regularity estimate
  \begin{equation}\label{eq:high-regularity}
    \| z \|_{H^{m+2}(\OchO)} \lesssim 
    \| g \|_{H^{m}(\OchO)}.
  \end{equation}
\end{lemma}

%Then, by combining Lemmas \ref{lem:high-regularity} and \ref{lem:stab}, we conclude the following lemma.
%\begin{lemma}\label{lem:stab-high}
%  Under the conditions of Lemma \ref{lem:conv}, let $w$ be the solution to \eqref{eq:PML3}. If the source $g\in H^m(\OchO)$ with $m \geq 0$, then $w \in H^{m+2}(\OchO) \cap H_0^1(\mcD)$ and satisfies
%  \begin{equation}\label{eq:stab-high}
%    \| w \|_{H^{m+2}(\OchO)} \lesssim k^{m+1} C_{m,g},
%  \end{equation}
%  where $C_{m,g}:=\sum_{j=0}^m k^{-j} \| g \|_{H^{j}(\OchO)}$. 
%\end{lemma}
%\begin{proof} 
%  The case of $m=0$ is obvious from Lemma \ref{lem:stab}. When $m=1$, by Lemma \ref{lem:high-regularity} with replacing the function $g$ with $Bk^2w+g$, we have
%  \begin{align*}
%    \|w\|_{H^3(\OchO)} \lesssim \|Bk^2w+g\|_{H^1(\OchO)} \lesssim k^2 \|g\|_{L^2(\mcD)} + \|g\|_{H^1(\OchO)}
%    \lesssim \cb{k^2 C_{1,g}.}
%  \end{align*}
%  Next, for a given integer $n\geq 2$, we suppose that \eqref{eq:stab-high} holds for all $m \leq n-1$.
%  Then, for $m=n$, by using Lemma \ref{lem:high-regularity} again we have 
%  \begin{align*}
%    \| w \|_{H^{n+2}(\OchO)} &\lesssim 
%    k^2 \|w\|_{H^{n}(\OchO)} + \|g\|_{H^{n}(\OchO)} % \\
%    % &\lesssim k^2 \sum_{j=0}^{n-2} k^{j+1} \|g\|_{H^{n-2-j}(\OchO)} + \|g\|_{H^{n}(\OchO)} \\
%    % &\lesssim \sum_{j=2}^{n} k^{j+1} \|g\|_{H^{n-j}(\OchO)} + \|g\|_{H^{n}(\OchO)} \\ 
%    % &\lesssim \sum_{j=0}^{n} k^{j+1} \|g\|_{H^{n-j}(\OchO)}.
%    \lesssim \cb{k^2(k^{n-1} C_{n-2,g}) + \|g\|_{H^{n}(\OchO)} 
%    \lesssim k^{n+1} C_{n,g}.}
%  \end{align*}
%  The proof is completed by induction.
%\end{proof}

\subsection{The elliptic operator}

Inspired by the insights of the recent study in \cite{gs23}, we start this section by introducing a truncated operator. This operator is constructed through a linear combination of a finite number of eigenfunctions related to a symmetrized version of the PDE operator of the truncated PML problem \eqref{eq:PML3}. 
Subsequently, we add some multiple of this truncated operator to the original PDE operator, ultimately yielding a positive definite elliptic operator.

Denote the real parts of $A$ and $B$ by $A_r(x) = \Re A(x)$ and $B_r(x) =\Re B(x)$, respectively. We introduce the elliptic operator $K: H_0^1(\mcD) \to H^{-1}(\mcD)$ defined by
\[
  \Dp{Ku,v} := (A_r \nabla u, \nabla v) - k^2(B_r u,v) \quad \forall v \in H_0^1(\mcD).
\]
Since $A_r$ is symmetric and positive definite (see \eqref{eq:Aqr}), by the spectral theorem (see, e.g., \cite{mclean00}), we write the eigenvalues and the associated eigenfunctions of $K$ by 
\[
  \lambda_1 \le \lambda_2 \le \cdots \quad\text{and}\quad \phi_1, \phi_2, \cdots,
\]
i.e., $K \phi_j = \lambda_j \phi_j$ in $H^{-1}(\mcD)$, where $\phi_j$ satisfies
\begin{equation}\label{eq:phi}
  -\Div(A_r \nabla \phi_j) - k^2 B_r \phi_j = \lambda_j \phi_j \;\;\mbox{in}\;D; \quad \phi_j = 0 \;\;\mbox{on}\;\hat\Gamma \quad \mbox{for } j=1,2,\cdots.
\end{equation}
In addition, $\{\phi_j\}_{j\geq 1}$ forms an orthonormal basis of $L^2(\mcD)$. 
% \cb{\sout{Denote by $c_B:=\|B_r\|_{L^\infty(\mcD)}$.}} 
By \eqref{eq:Aqr}--\eqref{eq:Bvw}, it holds
\eqn{
  \lambda_j \|\phi_j\|_{L^2(\mcD)}^2 = \Dp{K \phi_j, \phi_j} &= (A_r \nabla \phi_j, \nabla \phi_j) - k^2(B_r \phi_j, \phi_j) 
  \geq - k^2 \|\phi_j\|_{L^2(\mcD)}^2,
}
which yields
\eq{\label{eq:lower-bound}
  \lambda_j \geq -k^2 \quad\text{for all } j = 1, 2, \cdots. 
}
Let $N \ge 1$ be the integer satisfying
\eq{\label{eq:N}
  \lambda_{N} \leq 2 k^2 < \lambda_{N+1}. 
}
For any $u \in L^2(\mcD)$, there exists an expansion 
\[
  u = \sum_{j=1}^\infty u_j \phi_j \quad \mbox{with} \quad u_j = (u,\phi_j).
\]
We define the truncated operator $T_N$ by
\eq{\label{eq:TN}
  T_N u := \sum_{j=1}^N u_j \phi_j.
}
The following lemma says that $T_N$ is a bounded regularizing operator from $L^2(\mcD)$ to $H^m(\OchO)$ for any integer $m\geq 0$.
\begin{lemma}\label{lem:TN}
  For any $m=0,1,\cdots$, it holds
  \eq{\label{TNu}
    \|T_N u\|_{H^m(\OchO)} \lesssim k^m \|u\|_{L^2(\mcD)} \quad\forall u\in L^2(\mcD).}
\end{lemma}
\begin{proof}
  Denote $v=T_N u$.
  The case of $m=0$ is obvious by Plancherel theorem.
  When $m=1$, by \eqref{eq:Aqr}--\eqref{eq:Bvw}, the definition of $K$, %$\phi_j$ in \eqref{eq:phi}, 
  and \eqref{eq:N}, we have 
  \begin{align*}
    \|\nabla v\|_{L^2(\mcD)}^2 &\lesssim (A_r \nabla v, \nabla v) = \Dp{K v, v} + k^2 (B_r v, v) \\
    &\lesssim \sum_{j=1}^N \lambda_j u_j (\phi_j, v) + k^2 \|v\|_{L^2(\mcD)}^2\lesssim k^2 \|u\|_{L^2(\mcD)}^2.
  \end{align*}
  Hence, by noting $k \gtrsim 1$ we get 
  \eq{\label{vH1}\|v\|_{H^1(\mcD)} \lesssim k \|u\|_{L^2(\mcD)}.}  
  The definition of $K$ says that
  \eq{\label{Kw} -\Div (A_r \nabla w)=K w+k^2B_r w \;\;\text{for any $w\in H_0^1(\mcD)$}.}
  For $m\ge 2$, by taking $w=v$ in \eqref{Kw}, %and noting 
  % $$Kv = K(T_N u) = \sum_{j=1}^N \lambda_j u_j \phi_j \in H^{m-2}(\OchO),$$ 
  similarly to the higher regularity result in Lemma \ref{lem:high-regularity} (replacing $A$ with $A_r$), we have
  \eq{\label{vHn}\|v\|_{H^m(\OchO)} \lesssim\|K v\|_{H^{m-2}(\OchO)}+k^2\|v\|_{H^{m-2}(\OchO)}.}
  Next, we estimate $K v$ in $H^{m-2}$-norm. 
  First, for any $l\ge 0$, taking $w=K^lv$ in \eqref{Kw} again gives
  \eq{\label{Klv}\|K^l v\|_{H^n(\OchO)} \lesssim\|K^{l+1} v\|_{H^{n-2}(\OchO)}+k^2\|K^l v\|_{H^{n-2}(\OchO)} \;\; \text{for all } n \geq 2.}
  By noting $K^l v = \sum_{j=1}^N \lambda_j^l u_j \phi_j$ and $|\lambda_j| \lesssim k^2$ when $j\leq N$, it holds that 
  \eq{\label{KlvL2}\| K^l v\|_{L^2(\mcD)} = \bigg(\sum_{j=1}^N \lambda_j^{2l} |u_j|^2 (\phi_j, \phi_j) \bigg)^\frac12\lesssim k^{2l} \|u\|_{L^2(\mcD)}.}
  Since $K^lv=T_NK^lu$, from \eqref{vH1} and \eqref{KlvL2}, we also have 
  \eq{\label{KlvH1}\|K^lv\|_{H^1(\mcD)} \lesssim k \|K^lu\|_{L^2(\mcD)}\lesssim k^{2l+1} \|u\|_{L^2(\mcD)}.}
  Therefore, by using the recursive formula in \eqref{Klv} and \eqref{KlvL2}--\eqref{KlvH1}, it is easy to find that
  \eqn{\|K v\|_{H^{m-2}(\OchO)}\lesssim k^m\|u\|_{L^2(\mcD)}.}
  By plugging the above estimate into \eqref{vHn}, we get 
  \eqn{\|v\|_{H^m(\OchO)} \lesssim k^m\|u\|_{L^2(\mcD)}+k^2\|v\|_{H^{m-2}(\OchO)},} 
  which together with \eqref{TNu} with $m=0,1$, implies by induction that \eqref{TNu} holds for any $m\ge 2$.  This concludes the proof of this lemma.
\end{proof}

Next we introduce the elliptic operator and the corresponding    sesquilinear form as follows.
\eq{\label{LN}
  \mcL_N&:=\mcL+2k^2 T_N,\\
  b(u,v)&:=\Dp{\mcL_N u,v}= a(u,v) + 2k^2 (T_N u,v) \quad \forall u,v \in H_0^1(\mcD).\label{eq:buv}
}
The following lemma says that the sesquilinear form $b$ is continuous and coercive, so the corresponding elliptic PDE operator $\mcL_N$ is positive definite. %it fits into the framework of the Lax-Milgram lemma.
\begin{lemma}\label{lem:b}
  For any $u, v\in H_0^1(\mcD)$, there hold
  \begin{align}\label{est:b:cont-coer}
    |b(u,v)| \lesssim \TV{u}\TV{v} \quad \mbox{and}\quad
    \Re b(u,u)  \gtrsim \TV{u}^2.
  \end{align}
\end{lemma}
\begin{proof}
  The continuity is obvious by noting \eqref{est:a:cont} and Lemma~\ref{lem:TN} with $m=0$. 
  For any $u \in H_0^1(\mcD)$, we write the expansion of $u$ as $u=\sum_{j=1}^\infty u_j \phi_j$. 
  By using the lower bound estimate of $\lambda_j$ in \eqref{eq:lower-bound} and noting the truncation order in \eqref{eq:N}, we have
  \begin{align*}
    \Re b(u,u) &= (A_r\nabla u,\nabla u) - k^2(B_r u,u) + 2k^2 (T_N u, u) \\
    &= \sum_{j=1}^{\infty} u_j \Dp{K \phi_j, u} + 2k^2 \sum_{j=1}^{N} u_j (\phi_j, u) \\
    &= \sum_{j=1}^{N} (\lambda_j + 2k^2) u_j (\phi_j,u) + \sum_{j=N+1}^\infty \lambda_j u_j (\phi_j,u) \\
    &\geq k^2 \sum_{j=1}^\infty |u_j|^2 (\phi_j,\phi_j) \\
    &= k^2 \|u\|^2_{L^2(\mcD)}.
  \end{align*}
  Moreover, from \eqref{eq:Aqr}--\eqref{eq:Bvw}, we get
  \begin{align*}
    \Re b(u,u)\ge (A_r \nabla u, \nabla u)- k^2(B_r u,u)\gtrsim \|\nabla u\|_{L^2(\mcD)}^2- k^2 \|u\|^2_{L^2(\mcD)}.
  \end{align*}
  Combining the above two estimates yields the second inequality of \eqref{est:b:cont-coer}. 
\end{proof}

\begin{remark} 
  % \cc{Please remark on the difference of the definition of our elliptic operator from that of \cite{gs23}.}
  In \eqref{LN}, we define an elliptic operator $\mathcal{L}_N$ by adding an appropriate multiple of the truncated operator $T_N$ (i.e., $2k^2T_N$ in this paper) to the original PDE operator $\mathcal{L}$. 
  Such a construction is inspired by a general but complex definition in \cite[eq. (2.1)]{gs23}, which introduced a regularizing operator $\mathcal{S} = \psi(\mathcal{P})$ to ensure coercivity when added to $k^{-2}\mathcal{L}$, where $\mathcal{P}$ represents the self-adjoint operator associated with the sesquilinear form $\Re a$ (refer to \cite[(1.10)]{gs23}), and $\psi \in C^\infty_{\rm comp}$ satisfying $x+\psi(x) \geq 1$ for $x \geq \lambda_1(\mathcal{P})$, where $\lambda_1(\mathcal{P})$ is the smallest eigenvalue of the operator $\mathcal{P}$. 
  It is easy to verify that our definition is actually equivalent to setting $\mathcal{P} = k^{-2} K$ and explicitly defining $\psi(x) \in C^\infty(\mathbb{R})$ as a cut-off function such that
  \begin{align*}
    \supp \psi \subset (-2, k^{-2}\lambda_{N+1}), \quad \psi|_{\mathbb{R}} \geq 0, \quad  \psi|_{[-1,2]} = 2,
  \end{align*}
  %  By \eqref{eq:lower-bound} and \eqref{eq:N}, $\psi(x)$ satisfies the condition that $x+\psi(x) \geq 1$ for $x \geq k^{-2} \lambda_1$. 
  %  In addition, the bounded, self-adjoint, and 
  and hence, the  regularizing operator is given by $\mathcal{S} = \psi(k^{-2} K) = 2 T_N$.
  Our explicit and simple definition of the elliptic operator is crucial for a better understanding and further analysis, particularly in the context of preasymptotic error analysis for higher-order methods.
\end{remark}

The next lemma gives the stability and regularity estimates for the elliptic problem associated with the elliptic operator $\mcL_N$. %cf. Lemma \ref{lem:stab-high}.
\begin{lemma}\label{lem:LN}
  Suppose $g \in H^m(\OchO)$ with $m \geq 0$ and let $w$ be the solution to the elliptic problem 
  \begin{align}
    \mcL_N w =g \;\; \textnormal{in} \; \mcD; \quad w=0  \;\;\textnormal{on} \; \hat\Gamma. \label{eq:LNwg}
  \end{align}
  Then $w \in H^{m+2}(\OchO)\cap H_0^1(\mcD)$ uniquely exists and satisfies
  \begin{equation}\label{est:LNw}
    \TV{w} \lesssim k^{-1} \|g\|_{L^2(\mcD)} \quad\mbox{and}\quad \| w \|_{H^{m+2}(\OchO)} \lesssim k^mC_{m,g},
  \end{equation}
 where $C_{m,g}:=\sum_{j=0}^m k^{-j} \| g \|_{H^{j}(\OchO)}$.
\end{lemma}
\begin{proof}
  By \eqref{eq:LNwg}, $w \in H_0^1(\mcD)$ is the solution to the variational formulation
  \[
    b(w, v) = (g, v) \quad \forall v \in H_0^1(\mcD).
  \]
  Taking $v = w$ and using Lemma \ref{lem:b} yield
  \begin{align*}
    \TV{w}^2 \lesssim \Re b(w,w) = \Re (g,w) \lesssim k^{-1}\|g\|_{L^2(\mcD)} \TV{w},
  \end{align*}
  which implies the first inequality of \eqref{est:LNw}. Noting 
  \[
    (A \nabla w,\nabla v) = (g - 2k^2 T_N w + k^2 B w,v) \quad \forall v\in H_0^1(\mcD),
  \]
  and using Lemma \ref{lem:high-regularity} and Lemma \ref{lem:TN}, we find that
  \begin{align*}
    \|w\|_{H^{m+2}(\OchO)} &\lesssim \|g\|_{H^m({\OchO})} + k^{m+2} \|w\|_{L^2({\mcD})} + k^2 \|w\|_{H^m(\OchO)} \\
    &\lesssim k^mC_{m,g} + k^2 \|w\|_{H^m(\OchO)}, 
  \end{align*}
  which gives the second inequality of \eqref{est:LNw} by induction. 
\end{proof}

\subsection{Decomposition lemma}\label{sec:decomp}
In this subsection, we prove the following decomposition lemma for the PML solution to \eqref{eq:PML3}, which is a key ingredient in the error estimation for higher-order (CIP-)FEM (see also \cite{ms10,zw13,dw15}).
\begin{lemma}\label{lem:decomp} 
  Under the conditions of Lemma \ref{lem:conv}, suppose $g\in H^m(\OchO)$ with $m \geq 0$ and let $w$ be the solution to the PML problem \eqref{eq:PML3}. 
  Then there exists a splitting $w = w_\mcE + w_\mcA$ such that $w_\mcE \in H^{m+2}(\OchO)\cap H_0^1(\mcD)$ and $w_\mcA \in H^{j}(\OchO) \cap H_0^1(\mcD)~(\forall j \geq 0)$ satisfying
  \begin{align}
    &\| w_\mcE \|_{H^j(\OchO)} \lesssim k^{j-2} \|g\|_{L^2(\mcD)}, \;\; j=0,1,2, \;\; \text{and}\;\; \| w_\mcE \|_{H^{m+2}(\OchO)} \lesssim k^m C_{m,g}; \label{est:wE-wA}\\
    &\| w_\mcA \|_{H^j(\OchO)} \lesssim k^{j-1} \| g \|_{L^2(\mcD)}; \label{est:wE-wA1}\\
    &\| w \|_{H^{m+2}(\OchO)}\lesssim k^{m+1} \tilde{C}_{m,g}, \label{est:wE-wA2}
  \end{align}
  where $\tilde{C}_{m,g}:=\| g \|_{L^2(\mcD)}+k^{-1}C_{m,g}$ with $ C_{m,g}=\sum_{j=0}^m k^{-j} \| g \|_{H^{j}(\OchO)}$.
\end{lemma}
\begin{proof}
  Let $w_{\mcE} \in H_0^1(\mcD)$ and $w_{\mcA} \in H_0^1(\mcD)$ be the solutions to
  \begin{equation}\label{eq:wE-wA} 
    b(w_{\mcE},v) = (g,v) \quad \mbox{and}\quad 
    a(w_{\mcA},v) = 2k^2  (T_N w_{\mcE},v) \quad \forall v\in H_0^1(\mcD),
  \end{equation}
  respectively. It's easy to see that $a(w_{\mcE}+w_{\mcA}, v)=(g,v)$ for all $v\in H_0^1(\mcD)$, which implies that $w=w_{\mcE}+w_{\mcA}$ is the solution to the problem \eqref{eq:PML3}. 

  It is clear that \eqref{est:wE-wA} holds by using Lemma \ref{lem:LN} and \eqref{est:wE-wA2} is a direct consequence of \eqref{est:wE-wA}--\eqref{est:wE-wA1}. It remains to prove \eqref{est:wE-wA1}. From Lemmas~\ref{lem:stab} and \ref{lem:TN}, we have
  \eq{\label{wA}
    \TV{w_\mcA}\lesssim k^2 \|w_{\mcE}\|_{L^2(\mcD)} \lesssim \|g\|_{L^2(\mcD)},
  }
  that is, \eqref{est:wE-wA1} holds for $j=0,1$.  Noting from \eqref{LN} and \eqref{eq:wE-wA} that
  \eqn{\mcL_Nw_{\mcA}=2k^2 T_N(w_{\mcE}+w_{\mcA}).}
  Therefore, by Lemma \ref{lem:LN}, Lemma \ref{lem:TN}, and \eqref{wA}, we obtain for $j\ge 2$,
  \eqn{
    \|w_{\mcA}\|_{H^j(\OchO)} &\lesssim k^{j-2} \sum_{i=0}^{j-2} k^{-i} \|2 k^2  T_N(w_{\mcE}+w_{\mcA})\|_{H^{i}(\OchO)}\\
    &\lesssim k^j\big(\|w_{\mcE}\|_{L^2(\mcD)}+\|w_{\mcA}\|_{L^2(\mcD)}\big) \lesssim k^{j-1}\|g\|_{L^2(\mcD)}.
  }
  This completes the proof of this lemma.
\end{proof}

\begin{remark} \label{rmk:decomp} 
  \begin{enumerate}[\rm(i)]
    \item In next section, the decomposition lemma will be used for proving the preasymptotic error estimate of the higher-order (CIP-)FEM. In the splitting $w=w_\mcE+w_\mcA$ with regularity estimates \eqref{est:wE-wA}--\eqref{est:wE-wA1}, the $H^2$-norm of the elliptic part $w_\mcE$ is independent of the wave number $k$, which will not cause any problems in the error estimate. The second part $w_\mcA$ is more oscillatory than $w_\mcE$ but analytic, so it can take advantage of the strengths of higher-order elements. 
    \item The splitting of the solution was first obtained by an Fourier transform technique in \cite[Lemma 3.5]{ms10}, which considered the Helmholtz equation with DtN boundary condition. 
    By defining a sequence of solutions to the elliptic problems recursively, a new approach to build the splitting was presented in \cite[Theorem 1]{cn20a}. 
    % Recently, the authors in \cite{gs23} first constructed a regularizing operator $\mathcal{S}$ which leads to a generalization of the ``elliptic-projection'' argument (first proposed in \cite{zw13}), and then rebuilt the splitting of the solution to a wide variety of abstract Helmholtz-type problems by using the same iterative method as that used in \cite{cn20a}. 
    % Inspired by \cite{gs23}, we construct an explicit regularizing operator $T_N$ instead of $\mathcal{S}$ 
    In this paper, we give a new splitting for the PML solution by using the regularizing operator $T_N$. 
    One of the main improvements is that the new splitting is simpler to construct and easier to estimate compared to the previous approaches in \cite{ms10,cn20a,gs23}, since it is defined by only one elliptic problem and one PML problem, i.e., the equations in \eqref{eq:wE-wA}. 
    In addition, the splitting in Lemma \ref{lem:decomp} is obtained for the PML problem with parameters that are piecewise constants, which is not included in the previous works \cite{cn20a,gs23}.
    \item Another obvious approach for deriving the higher regularity estimate of the PDE solution $w$ is to apply the shift estimate in Lemma~\ref{lem:high-regularity} and induction (see also \cite{dw15,gs23}), which would give the estimate  $\| w \|_{H^{m+2}(\OchO)} \lesssim k^{m+1} C_{m,g}$. Clearly, our result \eqref{est:wE-wA2} based on the splitting improves the previous estimates. 
  \end{enumerate}
\end{remark}
% \begin{remark}
  
% \end{remark}

\section{CIP-FEM and preasymptotic error estimates}\label{sec:CIP-FEM}
In this section, we first introduce the CIP-FEM (including the FEM) for the PML problem \eqref{eq:PML1}--\eqref{eq:PML2}, and then present the preasymptotic error estimates for them of arbitrary but fixed order $p \geq 1$.

\subsection{The CIP-FEM}
Let $\{\mcT_h\}$ be a family of subdivisions of $\mcD$ whose members are curvilinear triangulations in two dimensions and curved tetrahedrons in three dimensions. 
For any $K\in \mcT_h$, we denote the diameter of $K$ by $h_K:=\diam (K)$, and for any $e \subset \partial K$, we denote $h_e:=\diam(e)$. The mesh size of $\mcT_h$ is defined by $h:=\max_{K\in\mcT_h} h_K$. We assume that $\mcT_h$ is quasi-uniform and regular, and satisfies $h \eqsim h_K \eqsim h_e$ for any $K \in \mcT_h$ and $h_e \subset \partial K$. 
For simplicity, we also assume that $\mcT_h$ fits the interface $\Gamma$, i.e., $\Gamma \cap {\rm int}\, K = \emptyset$ for any $K\in \mcT_h$.
Let $\mcE_h^I$ denote the set of all the interior edges/faces of $\mcT_h$ in $\Omega$.
Additionally, we denote by $\widehat K$ the reference element of $K \in \mcT_h$ and by $F_K$ the element map from $\widehat K$ to $K$; see, e.g., \cite[\S 5]{ms10}. 

The $p$-th order ($p\ge 1$) finite element space is defined by
\[
  V_h:=\big\{ v_h \in H_0^1(\mcD): v_h|_K \circ F_K \in \mcP_p(\wh K)\;\;\mbox{for all }K\in\mcT_h \big\},
\]
where $\mcP_p(\wh K)$ denotes the set of all polynomials with degrees do not exceed $p$ on $\wh K$. 
% The $p$-th order FEM is given by: find $u_h \in V_h$ such that
% \begin{equation}
%   a(u_h, v_h) = (f, v_h) \quad \forall  v_h \in V_h. \label{eq:FEM}
% \end{equation}
% where $a(\cdot, \cdot)$ is defined in \eqref{eq:auv}. From \eqref{eq:vp} and \eqref{eq:FEM}, the following Galerkin orthogonality holds
% \begin{equation}
%   a(\hu - u_h, v_h) = 0 \quad \forall  v_h \in V_h. \label{eq:ortho}
% \end{equation}
Next, we introduce the CIP-FEM for the PML problem \eqref{eq:PML1}--\eqref{eq:PML2}. First, we define the penalty term by
\begin{equation}
  J(v,w) := \sum_{j=1}^p \sum_{e \in \mcE_h^I} \gamma_{j,e} h_e^{2j-1} \Dp{\Jp{\frac{\partial^j v}{\partial n_e^j}},\Jp{\frac{\partial^j w}{\partial n_e^j}}}_e,
\end{equation}
where $\gamma_{j,e}$ are the penalty parameters. Denote by $V:=H_0^1(\mcD) \cap H^{p+1}(\mcT_h)$ and
\[
  a_\gamma(v, w) := a(v,w) + J(v,w) \quad \forall  v, w \in V,
\]
where $a(\cdot, \cdot)$ is defined in \eqref{eq:auv}. 
Then the $p$-th order CIP-FEM for \eqref{eq:PML1}--\eqref{eq:PML2} reads as: find $u_h \in V_h$ such that
\begin{equation}
  a_\gamma(u_h, v_h) = (f, v_h) \quad \forall  v_h \in V_h. \label{eq:FEM}
\end{equation}
\begin{remark}
{\rm (i)} Clearly, if $\gamma_{j,e} \equiv 0$, then the CIP-FEM becomes the standard FEM. 

{\rm (ii)} The CIP-FEM was first introduced by Douglas and Dupont \cite{dd76} for second-order elliptic and parabolic PDEs, and has been applied to the the Helmholtz problem \eqref{eq:Helm} with the impedance, DtN, and PML boundary conditions; see, e.g., \cite{wu14,zw13,dw15,lzz20,lw19,zw23}. The CIP-FEM has been proved to be highly effective in reducing pollution errors, while the work \cite{lw19} on PML boundary condition focused only on the linear CIP-FEM. 
%
%(iv) In this paper we consider the scattering problem with time dependence $e^{-\i\omega t}$, that is, the sign before $\i$ in \cref{eq:Somm} is negative. If we consider the scattering problem with time dependence $e^{\i\omega t}$, that is, the sign before $\i$ in  \cref{eq:Somm} is positive, then the penalty parameters should be complex numbers with  nonnegative imaginary parts.
\end{remark}

\subsection{Approximation properties}
% For further analysis, we first give some notations here. 
Define the following discrete energy norm by
\begin{equation}
  \TV{v}_\gamma := \bggS{\TV{v}^2 + \sum_{j=1}^p \sum_{e\in \mcE_h^I} |\gamma_{j,e}| h_e^{2j-1}\norm{\Jp{\frac{\partial^j v}{\partial n_e^j}}}_{L^2(e)}^{2}}^{1/2}. \label{eq:norm-g}
\end{equation}
Given that the above norm may not be well defined for functions not smooth enough, we introduce the following quantity to measure the error between functions $v \in H^1(\mcD)$ and $v_h \in V_h$.
\begin{equation}
  E_\gamma (v,v_h) := \bggS{\TV{v-v_h}^2 + \sum_{j=1}^p \sum_{e\in \mcE_h^I} |\gamma_{j,e}| h_e^{2j-1}\norm{\Jp{\frac{\partial^j v_h}{\partial n_e^j}}}_{L^2(e)}^{2}}^{1/2}. \label{eq:E-g}
\end{equation}
Clearly, $E_\gamma (v,v_h)=\TV{v-v_h}_\gamma$ if $v$ is sufficiently smooth.
% \noindent In the rest of this paper, we assume that $-\gamma_0 \leq \gamma_{j,e} \lesssim 1$.
We have the following approximation property for $V_h$, which follows from \cite[Lemma 6.1]{dw15}.
\begin{lemma}\label{lem:approx}
  Let $1 \leq s \leq p+1$ and  $v \in H^s(\OchO) \cap H_0^1(\mcD)$.  Suppose $kh \lesssim 1$ and $|\gamma_{j,e}|\lesssim 1$ for $1\le j\le p, e\in\mcE_h^I$. Then there exists $v_h \in V_h$ such that
  \begin{equation*}%\label{est:proj}
    \| v - v_h \|_{L^2(\mcD)} + h E_\gamma(v,v_h) \lesssim h^s\|v\|_{H^s(\OchO)}.
  \end{equation*}
\end{lemma}

The following lemma gives approximation estimates for the solution to the truncated PML problem \eqref{eq:PML3} in the discrete space $V_h$.
\begin{lemma}\label{lem:approx2}
  Suppose $g \in H^m(\OchO)$ with $m\geq 0$ and let $w\in H^{m+2}(\OchO) \cap H_0^1(\mcD)$ be the solution to PML problem \eqref{eq:PML3}. Suppose $kh \lesssim 1$ and $|\gamma_{j,e}|\lesssim 1$ for $1\le j\le p, e\in\mcE_h^I$. Then
  \begin{equation*}%\label{est:app1}
    \inf_{v_h \in V_h} E_\gamma(w,v_h)\lesssim
    k^{s-1} h^s C_{s-1,g} + (kh)^p \|g\|_{L^2(\mcD)},
    % \begin{cases}
    %   (h + (kh)^p) \|g\|_{L^2(\mcD)}&\text{ if } g\in L^2(\mcD),\\
    %   (kh)^p C_{p-1,g} &\text{ if } g\in H^{p-1}(\OchO),
    % \end{cases}
  \end{equation*}
  where $s = \min \{ m+1, p\}$ and $C_{s-1,g}=\sum_{j=0}^{s-1}k^{-j}\|g\|_{H^j(\OchO)}$.  
\end{lemma}
\begin{proof}
  By using Lemma \ref{lem:decomp} with $m=s-1$ and $n=p+1$, there exists a splitting $w = w_\mcE + w_\mcA$ satisfying regularities \eqref{est:wE-wA}.
  According to Lemma \ref{lem:approx}, there exist two functions $w_{\mcE,h}$ and $w_{\mcA,h}$ in $V_h$ such that
  \begin{align*}
    E_\gamma(w_\mcE, w_{\mcE,h}) &\lesssim h^s \|w_\mcE\|_{H^{s+1}(\OchO)} \lesssim k^{s-1}h^s C_{s-1,g}, \quad\mbox{and}\\
    E_\gamma(w_\mcA, w_{\mcA,h}) &\lesssim h^p \|w_\mcA\|_{H^{p+1}(\OchO)} \lesssim (kh)^p \|g\|_{L^2(\mcD)}.
  \end{align*}
  Denote by $w_h := w_{\mcE,h} + w_{\mcA,h}$ the approximation of $w$  in $V_h$. Then from the triangle inequality we arrive at the first estimate and conclude the proof of this lemma.
  \end{proof}

\subsection{The elliptic projections}
To analyze the error of the CIP-FEM, we need to modify the ``elliptic'' sesquilinear form $b$ in \eqref{eq:buv} by adding the penalty term $J$ to it: for any $v,w \in V$, let
\begin{equation}
  b_\gamma(v,w) := b(v,w) + J(v,w). \label{eq:bvw-h}
\end{equation}
By using the trace inequality and inverse estimate, and noting \eqref{est:b:cont-coer}, it is easy to prove the following property 
% ies for the sesquilinear forms and $\gamma$-norm 
(cf. \cite[Lemma 6.3]{dw15}).
\begin{lemma}\label{lem:g-norm}
  There exists a constant $\gamma_0 > 0$ such that if $-\gamma_0 \leq \gamma_{j,e} \lesssim 1$ for all $1\leq j \leq p$ and $e \in \mcE_h^I$, then for any $v_h,\,w_h \in V_h$, there hold
  \begin{align}
%    |J(v_h, w_h)| &\lesssim \|\nabla v_h\|_{L^2(\mcD)} \|\nabla w_h\|_{L^2(\mcD)}, \label{eq:Jvhwh}\\
%    |J(v_h,v_h)| &\eqsim \|\nabla v_h\|_{L^2(\mcD)}^2, \label{eq:Jvhvh}\\
    \Re b_\gamma(v_h, v_h) &\eqsim \TV{v_h}_\gamma^2 \eqsim \TV{v_h}^2 \eqsim \Re b(v_h,v_h). \label{eq:bhvhvh}
  \end{align}
\end{lemma}

Now we define the elliptic projection $P_h: H_0^1(\mcD) \to V_h$ by
\begin{align}
  b_\gamma(v_h, P_h v) = b(v_h, v) \quad \forall  v_h \in V_h. \label{eq:proj}
\end{align}
By using Lemma \ref{lem:g-norm} and the duality argument, we have the following estimates for the elliptic projection.
\begin{lemma}\label{lem:proj}
  Under the conditions of Lemma \ref{lem:g-norm} and $kh\lesssim 1$, there hold for $v\in H_0^1(\mcD)$,
  \begin{align}
    h^{-1}\|v - P_h v\|_{L^2(\mcD)} &\lesssim E_\gamma(v,P_h v) \lesssim \inf_{v_h \in V_h} E_\gamma(v, v_h), \label{eq:err-Phv}
  \end{align}
\end{lemma}
\begin{proof}
  By \eqref{eq:bvw-h} and \eqref{eq:proj},
  \begin{equation}
    b(v_h, v - P_h v) = J(v_h,P_h v) \quad \forall  v_h \in V_h. \label{eq:ortho-b}
  \end{equation}
  Denote by $\varphi_h = v_h - P_h v$. It follows from Lemma \ref{lem:g-norm}, \eqref{eq:norm-g}--\eqref{eq:E-g}, and the H{\"o}lder inequality that
  \begin{align*}
    \TV{\varphi_h}_\gamma^2 &\eqsim \Re b_\gamma(\varphi_h,\varphi_h)= \Re b(\varphi_h,v_h-v) + J(\varphi_h,v_h) \lesssim E_\gamma(v,v_h) \TV{\varphi_h}_\gamma,
  \end{align*}
  which yields the second inequality in \eqref{eq:err-Phv} by noting the triangle inequality 
  \[
    E_\gamma(v,P_h v) \leq E_\gamma(v,v_h) + \TV{v_h-P_h v}_\gamma.
  \]

  Next, we consider the dual problem: find $w \in H_0^1(\mcD)$ such that
  \begin{align}\label{eq:dual:Ph}
    b(w,\varphi) = (v-P_h v,\varphi) \quad \forall  \varphi \in H_0^1(\mcD).
  \end{align}
  By \eqref{est:LNw} in Lemma \ref{lem:LN} with $m=0$, the following regularity for $w$ holds
  \begin{align}\label{est:wH2}
    \|w\|_{H^2(\OchO)} \lesssim \|v-P_h v\|_{L^2(\mcD)}.
  \end{align}
  Let $\varphi = v-P_h v$ in \eqref{eq:dual:Ph}, by using Lemma \ref{lem:approx} (with $s=2$) and \eqref{est:wH2}, we obtain
  \begin{align}\label{est:Ph}
    \|v - P_h v\|_{L^2(\mcD)}^2 &= b(w, v-P_h v)\notag \\
    &= \inf_{v_h \in V_h} \big\{b(w-v_h, v - P_h v) + J(v_h,P_h v)\big\}\notag\\
    &\lesssim E_\gamma(v, P_h v) \inf_{v_h \in V_h} E_\gamma(w, v_h) \\
    &\lesssim E_\gamma(v, P_h v) h \|v-P_hv\|_{L^2(\mcD)} \notag
  \end{align}
  which yields the first inequality in \eqref{eq:err-Phv}. The proof of this lemma is completed.
\end{proof}

%Finally, we prove the following error estimates for the elliptic projections of the PML solution by combining Lemma \ref{lem:proj} and the decomposition result in Lemma \ref{lem:decomp}.
%\begin{lemma}\label{lem:proj:PML}
%  Under the conditions of Lemmas \ref{lem:conv} and \ref{lem:g-norm}, suppose $kh \lesssim 1$ and let $w\in H_0^1(\mcD)$ be the solution to \eqref{eq:PML3}, then
%  \begin{align*}
%    \|{w - P_h w}\|_{L^2({\mcD})} + h E_\gamma(w, P_h w) &\lesssim (h^2 + h(kh)^p) \|g\|_{L^2(\mcD)}.
%  \end{align*}
%\end{lemma}
%\begin{proof}
%  From Lemma \ref{lem:proj}, it suffices to prove that 
%  \begin{equation}
%    \inf_{v_h \in V_h} E_\gamma(w,v_h) \lesssim (h+(kh)^p)\|g\|_{L^2(\mcD)}. \label{eq:err:proj:PML_1}
%  \end{equation}
%\end{proof}

The following lemma gives an estimate of the $T_N$-truncation of the error of elliptic projection.
\begin{lemma}\label{lem:proj:PML2}
% Suppose $kh \lesssim 1$ and let $v\in H_0^1(\mcD)$. 
Under the conditions of Lemmas \ref{lem:conv} and \ref{lem:g-norm}. 
Suppose $kh \lesssim 1$ and let $v\in H_0^1(\mcD)$,
there holds
  \begin{align}
    \|T_N(v - P_h v)\|_{L^2({\mcD})} &\lesssim k^{p-1}h^p\inf_{v_h \in V_h} E_\gamma(v, v_h). \label{est:TN-Ph}
  \end{align}
\end{lemma}
\begin{proof}
  We consider the dual problem: Find $w\in H_0^1(\mcD)$ such that
  \[
    b(w,\varphi) = (T_N(v - P_h v),\varphi) \quad \forall \varphi \in H_0^1(\mcD).
  \]
  Then, similar to \eqref{est:Ph}, we have 
  \begin{align}
    \|T_N(v - P_h v)\|_{L^2({\mcD})}^2 &= b(w, v-P_h v) \notag \\
    &\lesssim E_\gamma(v,P_h v)\inf_{v_h \in V_h} E_\gamma(w, v_h)\notag\\
    &\lesssim \inf_{v_h \in V_h} E_\gamma(v, v_h) h^{p}\|w\|_{H^{p+1}(\OchO)}, \label{est:err-3}
  \end{align}
  where we have used Lemmas~\ref{lem:proj} and \ref{lem:approx} to derive the last inequality. 
  Moreover, by Lemma \ref{lem:LN} and Lemma \ref{lem:TN}, we get
  \begin{align*}
    \|w\|_{H^{p+1}(\OchO)} \lesssim k^{p-1} \sum_{j=0}^{p-1} k^{-j}\|T_N(v - P_h v)\|_{H^j(\OchO)} \lesssim k^{p-1} \|T_N(v - P_h v)\|_{L^2(\mcD)},
  \end{align*}
  which together with \eqref{est:err-3} implies \eqref{est:TN-Ph} and concludes the proof of this lemma.
\end{proof}

\subsection{Preasymptotic error estimates}
Next we give the preasymptotic error estimates for the CIP-FEM \eqref{eq:FEM} as follows.
\begin{theorem}\label{thm:err:FEM}
  Under the conditions of Lemmas \ref{lem:conv} and \ref{lem:g-norm}, let $\hu$ and $u_h$ be the truncated PML solution to \eqref{eq:PML1}--\eqref{eq:PML2} and the CIP-FE solution to \eqref{eq:FEM}, respectively. Then there exists a positive constant $C_p$ independent of $k$, $h$, and $f$, but may depend on $p$, such that if 
  \begin{equation}
    k^{2p+1} h^{2p} \leq C_p, \label{eq:cond:pre}
  \end{equation}
  then the following error estimates hold
  \begin{align}
    E_\gamma(\hu, u_h) &\lesssim (1 + k(kh)^{p}) \inf_{v_h \in V_h} E_\gamma(\hu, v_h), \label{eq:FEM:err-1} \\
    \|\hu-u_h\|_{L^2(\mcD)} &\lesssim (h + (kh)^p) \inf_{v_h \in V_h} E_\gamma(\hu, v_h). \label{eq:FEM:err-2}
  \end{align}
  Furthermore, if $f|_\Omega\in H^{p-1}(\Omega)$, then
  \begin{align}
    \|{u - u_h}\|_{H^1(\Omega)}
    &\lesssim (kh)^p \hat C_{p-1,f}^{(p)} + k(kh)^{2p} \|f\|_{L^2(\Omega)} + \mathcal{E}^{\rm PML} \|f\|_{L^2(\Omega)}\label{eq:FEM:err-1a},\\
    k\|{u - u_h}\|_{L^2(\Omega)}&\lesssim  (kh)^{p+1} \hat C_{p-1,f}^{(p-1)} + k(kh)^{2p} \|f\|_{L^2(\Omega)}+ \mathcal{E}^{\rm PML} \|f\|_{L^2(\Omega)}, \label{eq:FEM:err-2a}
  \end{align}
  where $\hat C^{(t)}_{s,f}:=(k^{-1}+(kh)^t)C_{s,f} + \|f\|_{L^2(\Omega)}$ with $C_{s,f}:=\sum_{j=0}^{s} k^{-j} \| f \|_{H^{j}(\Omega)}$ and $\mathcal{E}^{\rm PML}$ is defined in \eqref{eq:err-PML}.
\end{theorem}
\begin{proof} \eqref{eq:FEM:err-1a} (resp. \eqref{eq:FEM:err-2a}) is a direct consequence of Lemmas \ref{lem:approx2} and \ref{lem:conv} and \eqref{eq:FEM:err-1} (resp. \eqref{eq:FEM:err-2}). We omit the details.
The remainder of the proof is divided into two steps.

  {\it Step 1: Estimating the $H^1$-error by the $L^2$-error.} First, from \eqref{eq:vp} and \eqref{eq:FEM}, the following Galerkin orthogonality holds
  \begin{equation}
    a(\hu - u_h, v_h) = J(u_h, v_h) \quad \forall  v_h \in V_h, \label{eq:ortho}
  \end{equation}
  which together with \eqref{est:b:cont-coer} and \eqref{eq:buv} implies that 
  \begin{align*}
    \TV{\hu-u_h}^2 &\lesssim \Re b(\hu-u_h,\hu-u_h) \\
    &= \Re a(\hu-u_h,\hu-u_h) + 2k^2  \|T_N(\hu-u_h)\|_{L^2(\mcD)}^2 \\
    &= \Re a(\hu-u_h,\hu-v_h) + \Re J(u_h,v_h-u_h) + 2k^2 \|T_N(\hu-u_h)\|_{L^2(\mcD)}^2,
  \end{align*}
  hence,
  \begin{align*}
    \TV{\hu-u_h}^2 +\Re J(u_h, u_h) \lesssim E_\gamma(\hu,u_h) \inf_{v_h \in V_h} E_\gamma(\hu, v_h) + 2k^2 \|T_N(\hu-u_h)\|_{L^2(\mcD)}^2,
  \end{align*}
  which gives
  \begin{align}
    E_\gamma(\hu, u_h) \lesssim \inf_{v_h \in V_h} E_\gamma(\hu, v_h) + k \|T_N(\hu-u_h)\|_{L^2(\mcD)}. \label{est:err:E}
  \end{align}
  {\it Step 2: Estimating the $L^2$-error by using the modified duality argument \cite{zw13}.}
  Let  $z\in  H_0^1(\mcD)$ be the solution to the following dual problem
  \[
    a(v,z) = (v,\hu-u_h) \quad \forall v \in H_0^1(\mcD).
  \]
  Using \eqref{eq:ortho},  \eqref{eq:ortho-b}, and Lemmas~\ref{lem:proj} and \ref{lem:proj:PML2}, we have
  \begin{equation}
    \label{est:err-1}
    \begin{aligned}
      \|\hu-u_h\|_{L^2(\mcD)}^2 &= a(\hu-u_h, z-P_h z) + J(u_h,P_h z)  \\
      &= b(\hu-v_h,z-P_h z) + J(v_h,P_h z) - 2k^2  (T_N(\hu-u_h),z-P_h z)  \\
      &= b(\hu-v_h,z-P_h z) + J(v_h,P_h z) - 2k^2  (\hu-u_h,T_N(z-P_h z))  \\
      &\lesssim E_\gamma(z,P_h z) \inf_{v_h\in V_h} E_\gamma(\hu, v_h) + k^2 \|\hu-u_h\|_{L^2(\mcD)}\|T_N(z-P_h z)\|_{L^2(\mcD)} \\
      &\lesssim \Big(\inf_{v_h\in V_h} E_\gamma(\hu, v_h) + k^{p+1}h^p \|\hu-u_h\|_{L^2(\mcD)}\Big)\inf_{v_h \in V_h} E_\gamma(z, v_h).
    \end{aligned}
  \end{equation}
  Noting from Lemma \ref{lem:approx2} that
  \begin{align}
    \inf_{v_h \in V_h} E_\gamma(z, v_h) \lesssim (h+(kh)^p) \|\hu - u_h\|_{L^2(\mcD)}. \label{est:err-2}
  \end{align}
  By combining \eqref{est:err-1}--\eqref{est:err-2}, we get
  \[
    \|\hu - u_h\|_{L^2(\mcD)} \lesssim (h+(kh)^p) \inf_{v_h \in V_h} E_\gamma(\hu,v_h) + (k^{p+1}h^{p+1}+k^{2p+1}h^{2p}) \|\hu - u_h\|_{L^2(\mcD)},
  \]
  which implies that if $k^{2p+1}h^{2p}$ is sufficiently small, then
  \begin{align}
    \|\hu - u_h\|_{L^2(\mcD)} \lesssim (h+(kh)^p) \inf_{v_h \in V_h} E_\gamma(\hu,v_h), \label{est:err:L2}
  \end{align}
  which is \eqref{eq:FEM:err-2}. Moreover, substituting \eqref{est:err:L2} into \eqref{est:err:E} gives \eqref{eq:FEM:err-1}.
This completes the proof of the theorem.
\end{proof}

\begin{remark}\label{rmk:err:FEM}
  \begin{enumerate}[\rm(i)]
    \item The preasymptotic $H^1$-error in \eqref{eq:FEM:err-1a} is bounded by three terms. The first term $O((kh)^p)$ is the same of the order as the interpolation error. 
    The second term $O(k(kh)^{2p})$ refers to the pollution error, which is reduced as the order $p$ of elements increases. The last term $\mcE^{\rm PML}$ is from the PML truncation and is exponentially small.
    
    \item Let us take a closer look at the dependences on source $f$ in \eqref{eq:FEM:err-1a}. 
    Firstly, the coefficient of the pollution term $O(k(kh)^{2p})$ depends only on the $L^2$-norm of $f$. 
    Secondly, the coefficient of the term $O((kh)^p)$ depends on $\hat C_{p-1,f}^{(p)}$, which consists of three parts, i.e., $k^{-1} C_{p-1, f}$, $\|f\|_{L^2(\Omega)}$, and $(kh)^p C_{p-1, f}$. 
    The first and second parts %$k^{-1} C_{p-1, f}$ and $\|f\|_{L^2(\Omega)}$ 
    refer to the interpolations of elliptic part $\hat u_\mcE$ and analytic part $\hat u_\mcA$, respectively, and the third part refers to the ``pollution'' error of $\hat u_\mcE$, which is smaller than the interpolation of $\hat u_\mcE$ if $k(kh)^p \lesssim 1$. 
    In addition, it is clear that $\hat C_{p-1,f}^{(p)} \lesssim C_{p-1,f}$ if $kh\lesssim 1$, %$\hat C_{p-1,f}^{(p)} \lesssim \tilde C_{p-1,f}$ if $k(kh)^p\lesssim 1$, 
    and $\hat C_{p-1,f}^{(p)} \lesssim k^{-\frac12}C_{p-1,f} + \|f\|_{L^2(\Omega)}$ if $k(kh)^{2p}\lesssim 1$, in particular, if \eqref{eq:cond:pre} holds.
    Compared with the estimate obtained in \cite{gs23} for PML problems with $C^{p,1}\cap C^3$ scaling function, whose coefficients in both terms depend on $C_{p-1,f}$ (see \cite[Theorem~4.9 and the proof of Theorem~1.5]{gs23}). Clearly, due to our improved regularity estimate in \eqref{est:wE-wA2}, the dependences of our estimate on the source term are weaker.
    
    \item If $\gamma_{j,e} \equiv 0$, then CIP-FEM becomes FEM. The results in Theorem \ref{thm:err:FEM}  hold for FEM as well.
  \end{enumerate} 
\end{remark}

\section{Numerical experiments}\label{sec:test}
In this section we simulate the Helmholtz problem \eqref{eq:Helm}--\eqref{eq:Somm}, which is first truncated by the PML \eqref{eq:PML1}--\eqref{eq:PML2} and then discretized by the FEM and CIP-FEM \eqref{eq:FEM}. All the codes are written in \href{https://www.firedrakeproject.org/}{Firedrake} \cite{Rathgeber2016}. 

Let the source be the cut-off function $f=1$ in $\mcB_R$ and $f=0$ outside $\mcB_R$.
Then the exact solutions of the Helmholtz equation \eqref{eq:Helm}--\eqref{eq:Somm} in $\bbR^{2}$ (cf. \cite{lw19}) are given by
\begin{alignat*}{2}
  u &= 
  \begin{cases}
    \frac{\mbi \pi R}{2k} H_1^{(1)}(kR) J_0(k|x|) - \frac{1}{k^2} & \mbox{in }\mcB_R, \\
    \frac{\mbi \pi R}{2k} J_1(kR) H_0^{(1)}(k|x|) & \mbox{otherwise.}
  \end{cases} 
\end{alignat*}
% {\color{red}and
% \begin{alignat*}{2}
%   u &= 
%   \begin{cases}
%     \mbi R^2 h_1^{(1)}(kR) j_0(k|x|) - \frac{1}{k^2} & \mbox{in }\mcB_R, \\
%     \mbi R^2 j_1(kR) h_0^{(1)}(k|x|) & \mbox{otherwise}
%   \end{cases}
% \end{alignat*}
% for $d=3$.} 
The domain $\Omega=\mcB_1$ is considered and the thickness of PML is chosen as $L=1$.
A sample mesh with size $h \approx 0.268$ is shown in the left of Figure \ref{fig:PML2D}.
The PML parameter is set by $\sigma_0 = 4$ so that the condition \eqref{PML:cond} is satisfied as long as $k \geq 2.5$. The following penalty parameters,  which are obtained in \cite{han12,zhou23,zw23} by a dispersion analysis for Helmholtz equation in $\mathbb{R}^2$ on equilateral triangulations, are used in the CIP-FEM to reduce the pollution errors. %\cc{use simpler $\gamma$?}
\begin{subequations}\label{eq:opt}
\begin{alignat}{2}
  &\qquad\gamma^{\rm opt} = -\tfrac{\sqrt{3}}{24} - \tfrac{\sqrt{3}}{1728}(kh)^2, &\quad&\mbox{for }p = 1; \\
  &\begin{cases}&\gamma^{\rm opt}_1 = -\tfrac{\sqrt{3}}{60} - \tfrac{97\sqrt{3}}{40320}\Sp{\tfrac{kh}{2}}^2,  \\
  &\gamma^{\rm opt}_2 = -\tfrac{\sqrt{3}}{1920} + \tfrac{3\sqrt{3}}{71680}\Sp{\tfrac{kh}{2}}^2,
  \end{cases} &\quad&\mbox{for }p = 2; \\
  &\begin{cases}&\gamma^{\rm opt}_1 = -0.017265294884296 - 0.000478304250473\Sp{\tfrac{kh}{3}}^2,  \\
  &\gamma^{\rm opt}_2 = -0.000192140229447 + 0.000015577502211\Sp{\tfrac{kh}{3}}^2,  \\
  &\gamma^{\rm opt}_3 = -0.000001264275697 + 0.000000540251047\Sp{\tfrac{kh}{3}}^2,
  \end{cases} &\quad&\mbox{for }p = 3;  
\end{alignat}
\end{subequations}

In the first example, we let $k=100$ and choose a mesh satisfying $kh/p \approx 1.5$, the right graph of Figure \ref{fig:PML2D} plots the plane-sections of the real parts of the exact solution and discrete solutions with $y=0$ for $p=1,2$. It shows clearly that the PML solutions decay rapidly to near $0$ outside $\Omega$. In addition, the solution of FEM has a significant phase error, whereas the solution of CIP-FEM fits the exact solution much better than that of FEM, whether for linear element or second-order element. 
\begin{figure}[t]
  \centering
  \begin{minipage}{0.4\textwidth}
    \centering
    \includegraphics[height=5cm]{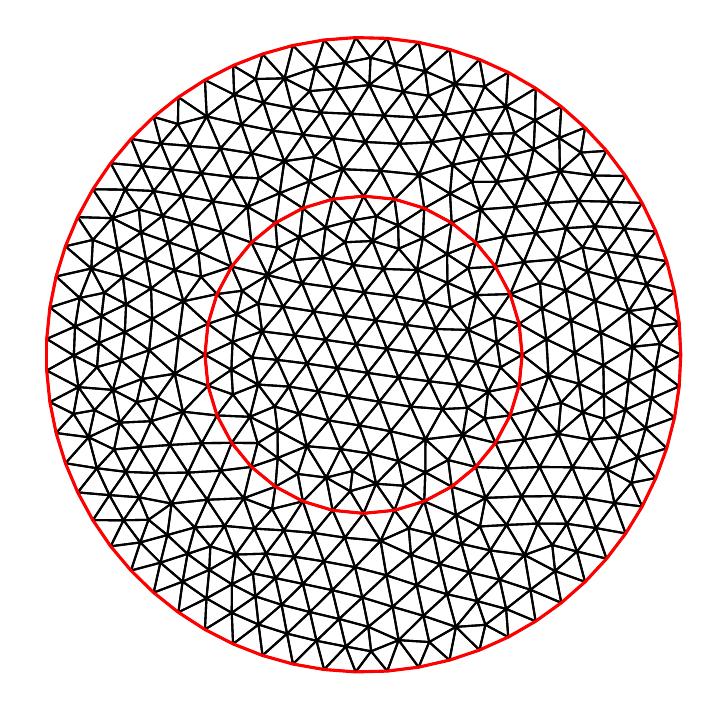}
  \end{minipage}
  \begin{minipage}{0.58\textwidth}
    \centering
    \includegraphics[height=5cm]{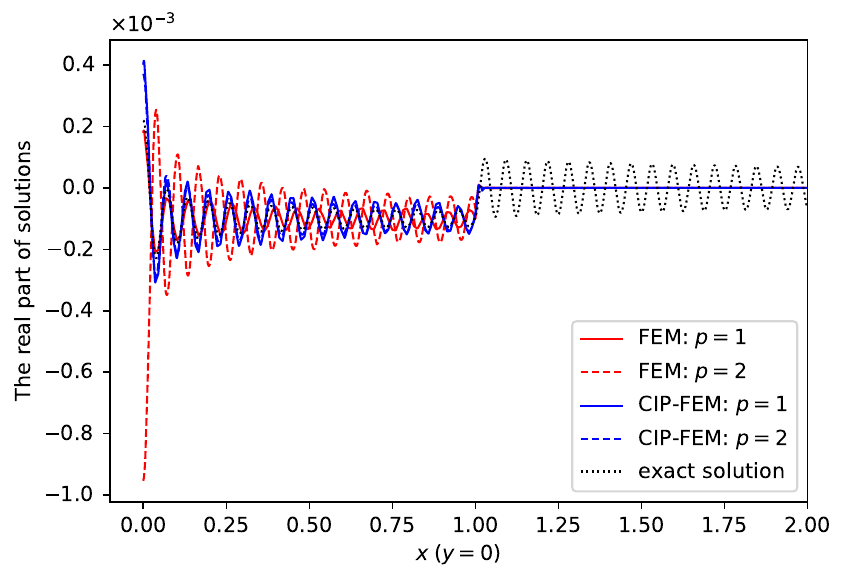}
  \end{minipage}
  \caption{Left: A sample mesh for the truncated PML problem. Right: The intersections of the real parts of exact solution and discrete solutions with the plane $y=0$ for $k=100$ and $h$ satisfying $kh/p \approx 1.5$.}
  \label{fig:PML2D}
\end{figure}

Next, we plot in Figure \ref{fig:test_hp} the relative $H^1$-errors of the FE and CIP-FE solutions, and the finite element interpolation for different $k$ and $p$, respectively. From these graphs, we find that the linear FEM ($p=1$) suffers from the pollution error when the mesh size is large (or, the degrees of freedom is small), especially for the case of large wave number (see the right graph). By increasing the order of finite elements, the pollution errors are significantly reduced, see the plots of FEM for $p=2$ and $p=3$. 
On the other hand, the graphs of Figure \ref{fig:test_hp} show that the performance of CIP-FEM is similar to that of FEM for a smaller wave number $k=20$, but, for $k=150$, the pollution range of the former is much smaller than that of the latter, whether for the linear element or higher-order elements. This behavior indicates that the CIP-FEM indeed reduces the pollution errors. 
\begin{figure}[t]
  \centering
  \includegraphics[width=0.9\textwidth]{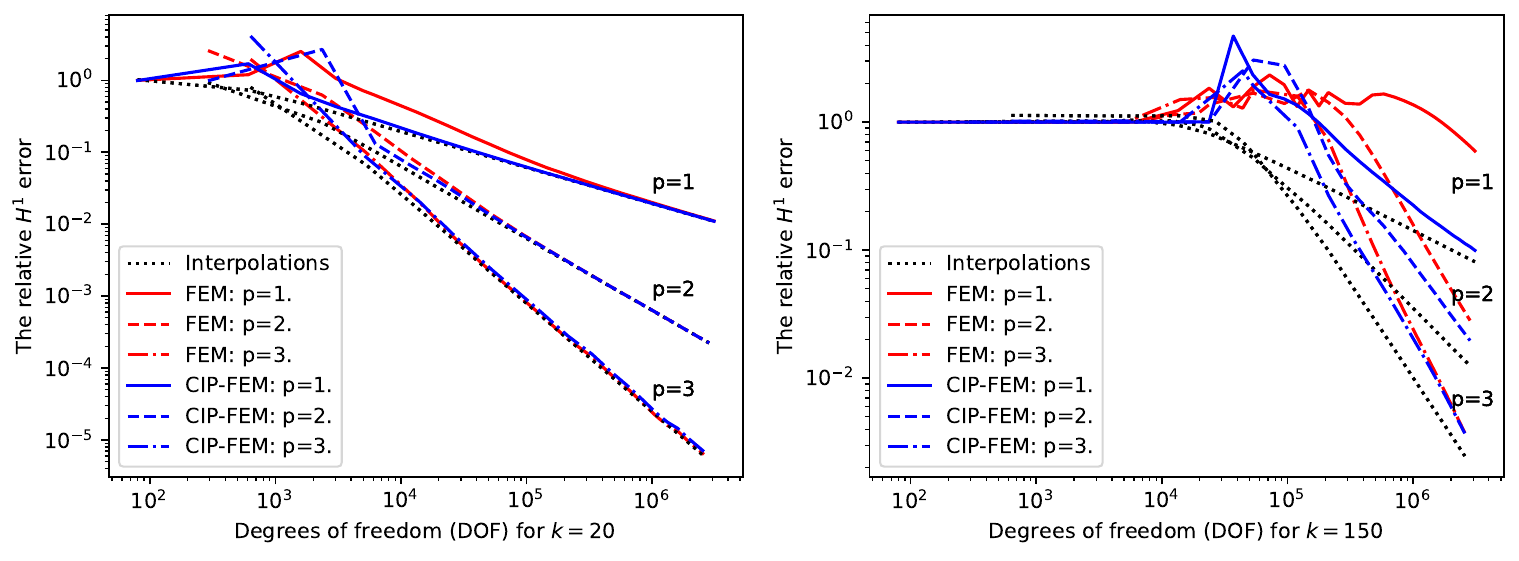}
  \caption{The relative $H^1$-errors of FEM and CIP-FEM, compared with the relative $H^1$-error of the finite interpolation (dotted) for $k=20$ (left graph), and $k=150$ (right graph) with orders $p=1,2,3$, respectively. %\cc{Regroup the plots by $k=20$ (left) and $k=150$ (right)?}
  }
  \label{fig:test_hp}
\end{figure}

In Figure \ref{fig:test_fix_khp}, we fix $kh/p \approx 1.2$ with $p=1,2,3$, and plot the errors of FEM, CIP-FEM and finite element interpolation as the wave number $k$ increases from $5$ to $450$. 
The potential of higher-order CIP-FEM in reducing the pollution error of Helmholtz problem with large wave numbers is clearly shown in this figure again. 
As $k$ increases, the pollution effects appear for linear FEM when the wave number $k$ is approximately greater than $20$, for second-order FEM ($p=2$) when $k \geq 100$, and for third-order FEM ($p=3$) when $k \geq 300$,
while the pollution errors for CIP-FEMs ($p=1,2,3$) are barely visible for $k$ up to $450$. 
% \cc{Why there is no $p=3$?}
\begin{figure}[t]
  \centering
  \includegraphics[height=6.5cm]{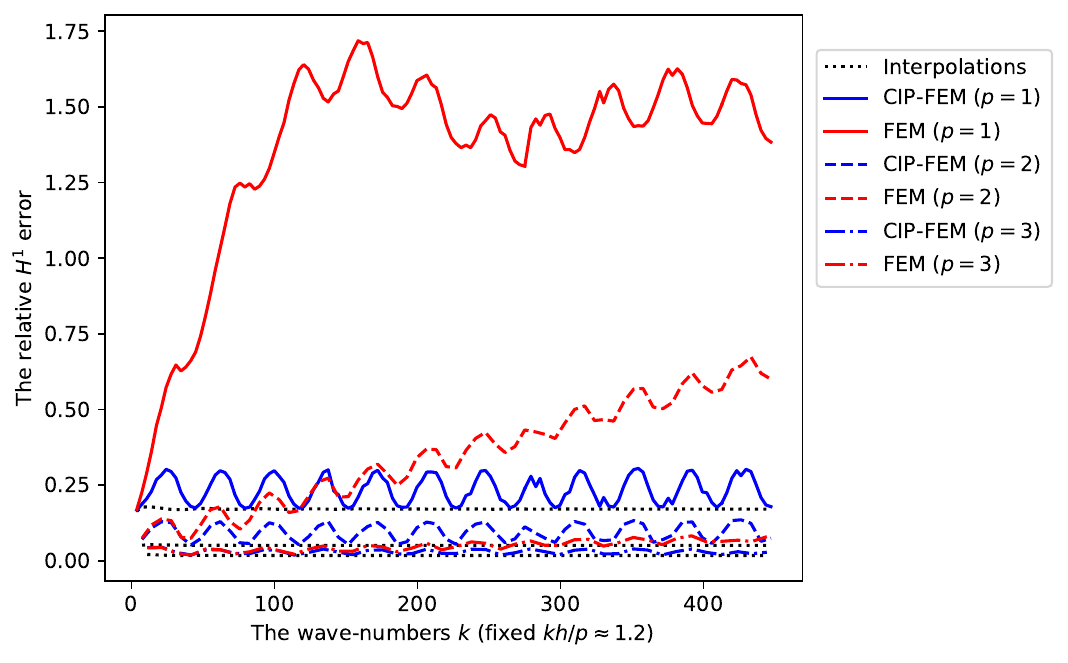}
  \caption{The relative $H^1$-errors of the FEM, CIP-FEM, and the finite element interpolation.}
  \label{fig:test_fix_khp}
\end{figure}

% \section*{Acknowledgments}
% The research of Yonglin Li was partially supported by the CAS AMSS-PolyU Joint Laboratory of Applied Mathematics.

%%%%%%%% Appendix %%%%%%%%
% \appendix
% \section{Appendix} 
%  ***

%%%%%%%%%%%%%%%%%%%%%%% Documents End %%%%%%%%%%%%%%%%%%%%%%%%%%

\bibliographystyle{abbrv} % plain, abbrv, siam, ...
\bibliography{references}

\begin{thebibliography}{10}

\bibitem{babuska70}
I.~Babu{\v s}ka.
\newblock The finite element method for elliptic equations with discontinuous
  coefficients.
\newblock {\em Computing}, 5(3):207--213, 1970.

\bibitem{bs00}
I.~M. Babu{\v s}ka and S.~A. Sauter.
\newblock Is the pollution effect of the {{FEM}} avoidable for the
  {{Helmholtz}} equation considering high wave numbers?
\newblock {\em SIAM Rev.}, 42(3):451--484, 2000.

\bibitem{bw05}
G.~Bao and H.~Wu.
\newblock Convergence analysis of the perfectly matched layer problems for
  time-harmonic {{Maxwell}}'s equations.
\newblock {\em SIAM J. Numer. Anal.}, 43(5):2121--2143, 2005.

\bibitem{berenger94}
J.-P. B{\'e}renger.
\newblock A perfectly matched layer for the absorption of electromagnetic
  waves.
\newblock {\em J. Comput. Phys.}, 114(2):185--200, 1994.

\bibitem{bp06}
J.~H. Bramble and J.~E. Pasciak.
\newblock Analysis of a finite {{PML}} approximation for the three dimensional
  time-harmonic {{Maxwell}} and acoustic scattering problems.
\newblock {\em Math. Comp.}, 76(258):597--614, 2006.

\bibitem{bp13}
J.~H. Bramble and J.~E. Pasciak.
\newblock Analysis of a {{Catesian PML}} approximation to acoustic scattering
  problems in $\mathbb{R}^2$ and $\mathbb{R}^3$.
\newblock {\em J. Comput. Appl. Math.}, 247, 2013.

\bibitem{bs08}
S.~C. Brenner and L.~R. Scott.
\newblock {\em The {{Mathematical Theory}} of {{Finite Element Methods}}}.
\newblock {Springer-Verlag}, 3rd edition, 2008.

\bibitem{bgp17}
D.~L. Brown, D.~Gallistl, and D.~Peterseim.
\newblock Multiscale {{Petrov-Galerkin}} method for high-frequency
  heterogeneous {{Helmholtz}} equations.
\newblock In {\em Meshfree {{Methods Partial Differ}}. {{Equ}}. {{VIII}}},
  Lect. Notes Comput. Sci. Eng., pages 85--115. {Springer International
  Publishing}, 2017.

\bibitem{cgn+22}
T.~{Chaumont-Frelet}, D.~Gallistl, S.~Nicaise, and J.~Tomezyk.
\newblock Wavenumber-explicit convergence analysis for finite element
  discretizations of time-harmonic wave propagation problems with perfectly
  matched layers.
\newblock {\em Commun. Math. Sci.}, 20(1):1--52, 2022.

\bibitem{cn20a}
T.~{Chaumont-Frelet} and S.~Nicaise.
\newblock Wavenumber explicit convergence analysis for finite element
  discretizations of general wave propagation problems.
\newblock {\em IMA Journal of Numerical Analysis}, 40(2):1503--1543, 2020.

\bibitem{cl05}
Z.~Chen and X.~Liu.
\newblock An adaptive perfectly matched layer technique for time-harmonic
  scattering problems.
\newblock {\em SIAM J. Numer. Anal.}, 43(2):645--671, 2005.

\bibitem{cw03}
Z.~Chen and H.~Wu.
\newblock An adaptive finite element method with perfectly matched absorbing
  layers for the wave scattering by periodic structures.
\newblock {\em SIAM J. Numer. Anal.}, 41(3):799--826, 2003.

\bibitem{cw08}
Z.~Chen and X.~Wu.
\newblock An adaptive uniaxial perfectly matched layer method for time-harmonic
  scattering problems.
\newblock {\em Numer. Math. Theor. Meth. Appl.}, 1(2):113--137, 2008.

\bibitem{cx13}
Z.~Chen and X.~Xiang.
\newblock A source transfer domain decomposition method for {{Helmholtz}}
  equations in unbounded domain.
\newblock {\em SIAM J. Numer. Anal.}, 51(4):2331--2356, 2013.

\bibitem{cz10}
Z.~Chen and W.~Zheng.
\newblock Convergence of the uniaxial perfectly matched layer method for
  time-harmonic scattering problems in two-layered media.
\newblock {\em SIAM J. Numer. Anal.}, 48(6):2158--2185, 2010.

\bibitem{cjm97}
W.~C. Chew, J.-M. Jin, and E.~Michielssen.
\newblock Complex coordinate stretching as a generalized absorbing boundary
  condition.
\newblock {\em Microw. Opt. Technol. Lett.}, 15:363--369, 1997.

\bibitem{cw94}
W.~C. Chew and W.~H. Weedon.
\newblock A {{3D}} perfectly matched medium from modified {{Maxwell}}'s
  equations with stretched coordinates.
\newblock {\em Microw. Opt. Technol. Lett.}, 7(13):599--604, 1994.

\bibitem{ciarlet78}
P.~G. Ciarlet.
\newblock {\em The {{Finite Element Method}} for {{Elliptic Problems}}},
  volume~4 of {\em Studies in {{Mathematics}} and Its {{Applications}}}.
\newblock {North-Holland}, {New York}, 1st edition, 1978.

\bibitem{cm98}
F.~Collino and P.~Monk.
\newblock The perfectly matched layer in curvilinear coordinates.
\newblock {\em SIAM J. Sci. Comput.}, 19(6):2061--2090, 1998.

\bibitem{dbb99}
A.~Deraemaeker, I.~Babu{\v s}ka, and P.~Bouillard.
\newblock Dispersion and pollution of the {{FEM}} solution for the
  {{Helmholtz}} equation in one, two and three dimensions.
\newblock {\em Int. J. Numer. Meth. Engng.}, 46(4):471--499, 1999.

\bibitem{dd76}
J.~Douglas and T.~Dupont.
\newblock {Interior Penalty Procedures for Elliptic and Parabolic Galerkin
  Methods}.
\newblock {\em Lecture Notes in Physics}, 58:207--216, 1976.

\bibitem{dw15}
Y.~Du and H.~Wu.
\newblock Preasymptotic error analysis of higher order {{FEM}} and {{CIP-FEM}}
  for {{Helmholtz}} equation with high wave number.
\newblock {\em SIAM J. Numer. Anal.}, 53(2):782--804, 2015.

\bibitem{fw09}
X.~Feng and H.~Wu.
\newblock Discontinuous {{Galerkin}} methods for the {{Helmholtz}} equation
  with large wave numbers.
\newblock {\em SIAM J. Numer. Anal.}, 47(4):2872--2896, 2009.

\bibitem{fw11}
X.~Feng and H.~Wu.
\newblock $hp$-discontinuous {{Galerkin}} methods for the {{Helmholtz}}
  equation with large wave number.
\newblock {\em Math. Comp.}, 80(276):1997--2024, 2011.

\bibitem{gls21}
J.~Galkowski, D.~Lafontaine, and E.~A. Spence.
\newblock Perfectly-{{Matched-Layer}} truncation is exponentially accurate at
  high frequency.
\newblock {\em ArXiv Prepr. ArXiv210507737}, 2021.

\bibitem{gls+21}
J.~Galkowski, D.~Lafontaine, E.~A. Spence, and J.~Wunsch.
\newblock Decompositions of high-frequency {{Helmholtz}} solutions via
  functional calculus, and application to the finite element method.
\newblock {\em ArXiv Prepr. ArXiv210213081}, 2021.

\bibitem{gls+22}
J.~Galkowski, D.~Lafontaine, E.~A. Spence, and J.~Wunsch.
\newblock The $hp$-{{FEM}} applied to the {{Helmholtz}} equation with {{PML}}
  truncation does not suffer from the pollution effect.
\newblock {\em ArXiv Prepr. ArXiv220705542}, 2022.

\bibitem{gs23}
J.~Galkowski and E.~A. Spence.
\newblock Sharp preasymptotic error bounds for the {{Helmholtz}} $h$-{{FEM}}.
\newblock {\em ArXiv Prepr. ArXiv230103574}, 2023.

\bibitem{gp15}
D.~Gallistl and D.~Peterseim.
\newblock Stable multiscale {{Petrov}}\textendash{{Galerkin}} finite element
  method for high frequency acoustic scattering.
\newblock {\em Comput. Methods Appl. Mech. Engrg.}, 295:1--17, 2015.

\bibitem{gh14}
C.~J. Gittelson and R.~Hiptmair.
\newblock Dispersion analysis of plane wave discontinuous {{Galerkin}} methods.
\newblock {\em Int. J. Numer. Meth. Engng.}, 98(5):313--323, 2014.

\bibitem{ghp09}
C.~J. Gittelson, R.~Hiptmair, and I.~Perugia.
\newblock Plane wave discontinuous {{Galerkin}} methods: Analysis of the
  $h$-version.
\newblock {\em ESAIM: M2AN}, 43(2):297--331, 2009.

\bibitem{gm11}
R.~Griesmaier and P.~Monk.
\newblock Error analysis for a hybridizable discontinuous {{Galerkin}} method
  for the {{Helmholtz}} equation.
\newblock {\em J. Sci. Comput.}, 49(3):291--310, 2011.

\bibitem{han12}
C.~Han.
\newblock Dispersion analysis of the {{IPFEM}} for the {{Helmholtz}} equation
  with high wave number on equilateral triangular meshes.
\newblock Master's thesis, Nanjing University, 2012.

\bibitem{hmp11}
R.~Hiptmair, A.~Moiola, and I.~Perugia.
\newblock Plane wave discontinuous {{Galerkin}} methods for the {{2D
  Helmholtz}} equation: Analysis of the $p$-version.
\newblock {\em SIAM J. Numer. Anal.}, 49(1):264--284, 2011.

\bibitem{hmp14}
R.~Hiptmair, A.~Moiola, and I.~Perugia.
\newblock Trefftz discontinuous {{Galerkin}} methods for acoustic scattering on
  locally refined meshes.
\newblock {\em Appl. Numer. Math.}, 79:79--91, 2014.

\bibitem{hmp16}
R.~Hiptmair, A.~Moiola, and I.~Perugia.
\newblock Plane wave discontinuous {{Galerkin}} methods: Exponential
  convergence of the $hp$-version.
\newblock {\em Found. Comput. Math.}, 16(3):637--675, 2016.

\bibitem{hmp16a}
R.~Hiptmair, A.~Moiola, and I.~Perugia.
\newblock A survey of {{Trefftz}} methods for the {{Helmholtz}} equation.
\newblock In {\em Building {{Bridges}}: {{Connections}} and {{Challenges}} in
  {{Modern Approaches}} to {{Numerical Partial Differential Equations}}}, pages
  237--279. {Springer}, 2016.

\bibitem{hm07}
T.~Huttunen and P.~Monk.
\newblock The use of plane waves to approximate wave propagation in anisotropic
  media.
\newblock {\em J. Comput. Math.}, 25(3):350--367, 2007.

\bibitem{ib95}
F.~Ihlenburg and I.~Babu{\v s}ka.
\newblock Finite element solution of the {{Helmholtz}} equation with high wave
  number. {{I}}. {{The}} $h$-version of the {{FEM}}.
\newblock {\em Comput. Math. Appl.}, 30(9):9--37, 1995.

\bibitem{ib97}
F.~Ihlenburg and I.~Babu{\v s}ka.
\newblock Finite element solution of the {{Helmholtz}} equation with high wave
  number part {{II}}: The $hp$ version of the {{FEM}}.
\newblock {\em SIAM J. Numer. Anal.}, 34(1):315--358, 1997.

\bibitem{jlw+22}
R.~Jiang, Y.~Li, H.~Wu, and J.~Zou.
\newblock Finite element method for a nonlinear perfectly matched layer
  {{Helmholtz}} equation with high wave number.
\newblock {\em SIAM J. Numer. Anal.}, 60(5):2866--2896, 2022.

\bibitem{km18}
S.~Kapita and P.~Monk.
\newblock A plane wave discontinuous {{Galerkin}} method with a
  {{Dirichlet-to-Neumann}} boundary condition for the scattering problem in
  acoustics.
\newblock {\em J. Comput. Appl. Math.}, 327:208--225, 2018.

\bibitem{lsw22}
D.~Lafontaine, E.~A. Spence, and J.~Wunsch.
\newblock Wavenumber-explicit convergence of the $hp$-{{FEM}} for the
  full-space heterogeneous {{Helmholtz}} equation with smooth coefficients.
\newblock {\em Comput. Math. Appl.}, 113:59--69, 2022.

\bibitem{ls98}
M.~Lassas and E.~Somersalo.
\newblock On the existence and convergence of the solution of {{PML}}
  equations.
\newblock {\em Computing}, 60(3):229--241, 1998.

\bibitem{ls01}
M.~Lassas and E.~Somersalo.
\newblock Analysis of the {{PML}} equations in general convex geometry.
\newblock {\em Proc. - R. Soc. Edinburgh, Sect. A: Math.}, 131(5):1183--1207,
  2001.

\bibitem{lw19}
Y.~Li and H.~Wu.
\newblock {{FEM}} and {{CIP-FEM}} for {{Helmholtz}} equation with high wave
  number and perfectly matched layer truncation.
\newblock {\em SIAM J. Numer. Anal.}, 57(1):96--126, 2019.

\bibitem{lzz20}
Y.~Li, W.~Zheng, and X.~Zhu.
\newblock A {{CIP-FEM}} for high-frequency scattering problem with the
  truncated {{DtN}} boundary condition.
\newblock {\em CSIAM Trans. Appl. Math.}, 1(3):530--560, 2020.

\bibitem{mclean00}
W.~McLean.
\newblock {\em Strongly Elliptic Systems and Boundary Integral Equations}.
\newblock {Cambridge university press}, 2000.

\bibitem{mps13}
J.~M. Melenk, A.~Parsania, and S.~A. Sauter.
\newblock General {{DG-methods}} for highly indefinite {{Helmholtz}} problems.
\newblock {\em J. Sci. Comput.}, 57(3):536--581, 2013.

\bibitem{ms10}
J.~M. Melenk and S.~A. Sauter.
\newblock Convergence analysis for finite element discretizations of the
  {{Helmholtz}} equation with {{Dirichlet-to-Neumann}} boundary conditions.
\newblock {\em Math. Comp.}, 79(272):1871--1914, 2010.

\bibitem{ms11}
J.~M. Melenk and S.~A. Sauter.
\newblock Wavenumber explicit convergence analysis for {{Galerkin}}
  discretizations of the {{Helmholtz}} equation.
\newblock {\em SIAM J. Numer. Anal.}, 49(3):1210--1243, 2011.

\bibitem{peterseim16}
D.~Peterseim.
\newblock Eliminating the pollution effect in {{Helmholtz}} problems by local
  subscale correction.
\newblock {\em Math. Comp.}, 86(305):1005--1036, 2016.

\bibitem{Rathgeber2016}
F.~Rathgeber, D.~A. Ham, L.~Mitchell, M.~Lange, F.~Luporini, A.~T.~T. McRae,
  G.-T. Bercea, G.~R. Markall, and P.~H.~J. Kelly.
\newblock Firedrake: automating the finite element method by composing
  abstractions.
\newblock {\em ACM Trans. Math. Softw.}, 43(3):24:1--24:27, 2016.

\bibitem{ty98}
E.~Turkel and A.~Yefet.
\newblock Absorbing {{PML}} boundary layers for wave-like equations.
\newblock {\em Appl. Numer. Math.}, 27(4):533--557, 1998.

\bibitem{wu14}
H.~Wu.
\newblock Pre-asymptotic error analysis of {{CIP-FEM}} and {{FEM}} for the
  {{Helmholtz}} equation with high wave number. {{Part I}}: Linear version.
\newblock {\em IMA J. Numer. Anal.}, 34(3):1266--1288, 2014.

\bibitem{zhou23}
Y.~Zhou.
\newblock {\em Dispersion analysis of {{CIP-FEM}} for {H}elmholtz problem}.
\newblock PhD thesis, Nanjing University, 2023, pre-print.

\bibitem{zw23}
Y.~Zhou and H.~Wu.
\newblock Dispersion analysis of {CIP-FEM} for the {H}elmholtz equation.
\newblock {\em SIAM J. Numer. Anal.}, 61(3):1278--1292, 2023.

\bibitem{zw21}
B.~Zhu and H.~Wu.
\newblock Preasymptotic error analysis of the {{HDG}} method for {{Helmholtz}}
  equation with large wave number.
\newblock {\em J. Sci. Comput.}, 87(2):1--34, 2021.

\bibitem{zw13}
L.~Zhu and H.~Wu.
\newblock Preasymptotic error analysis of {{CIP-FEM}} and {{FEM}} for
  {{Helmholtz}} equation with high wave number. {{Part II}}: $hp$ version.
\newblock {\em SIAM J. Numer. Anal.}, 51(3):1828--1852, 2013.

\end{thebibliography}

\end{document}